\documentclass[11pt]{amsart}     
\usepackage{amssymb,amsthm,graphicx,here}
\newtheorem{theorem}{Theorem}

\newtheorem{proposition}{Proposition}
\newtheorem{lemma}{Lemma}

\newcommand{\ord}{{\rm{ord}}}
\begin{document}
\title[\text{the fundamental group of linear torus curves of maximal contact}]
{On the fundamental group of the complement of linear torus curves
of maximal contact}
\author{M. Kawashima}
\address{\vtop{
\hbox{Department of Mathematics}
\hbox{Tokyo University of science}
\hbox{wakamiya-cho 26, shinjuku-ku}
\hbox{Tokyo 162-0827 Japan}
\hbox{\rm{E-mail}: {\rm j1107702@ed.kagu.tus.ac.jp}}
}}
\begin{abstract}
In this paper, we compute the fundamental group  of
the complement of linear torus curves of maximal contact
  and we show that
it is isomorphic that of generic linear torus curves.
\end{abstract}
\maketitle 
\section{introduction}
A plane curve $C\subset \Bbb P^2$ of degree $pq$ is called a 
{\em curve  of $(p,q)$ torus  type}  with $p>  q\ge  2$, 
if there is a defining polynomial $F$ of $C$
of the form  $F=F_{p}^q-F_{q}^p$,
where $F_p$, $F_q$ are homogeneous polynomials of $X,Y,Z$
of degree $p$ and $q$ respectively. 
A singular point $P\in C$ is called {\em inner} if $P\in \{F_p=F_q=0\}$.
Otherwise, $P$ is called an {\em outer singularity}.
$C$ is called {\em tame} if $C$ has no outer singularity.

We say that a $(p,q)$ torus curve $C$ is
{\em{a torus curve of a maximal contact}} 
if $\{F_p=F_q=0\}=\{\xi_0\}$ and $\{F_p=0\}$ is smooth at $\xi_0$.
Then the intersection multiplicity $I(F_q,F_p;\xi_0)$ is 
equal to $pq$ by B\'ezout's theorem and singularity $(C,\xi_0)$ 
is topologically equivalent to
the Brieskorn-Pham singularity $B_{p^2q,q}$
where $B_{m,n}:=\{(x,y)\in \Bbb{C}^2\,|\, x^m+y^n=0\}$.

We are interested in topological invariants
the Alexander polynomial $\Delta_C(t)$ and
the fundamental groups
$\pi_1(\Bbb{P}^2\setminus C)$ and
$\pi_1(\Bbb{C}^2\setminus C)$.
It is known that 
if  $C$ is a tame maximal contact of type $(p,q)$,
then the Alexander polynomial of $C$ is equal to
\[
 \Delta_{C}(t)=\dfrac{(t^{pq/r}-1)^r(t-1)}{(t^p-1)(t^q-1)}
\]
where $r={\text{gcd$(p,q)$}}$ (\cite{BenoitTu,BenoitTuCorrect}).   

We consider the special class of torus curves of type $(p,q)$.
We say that  $C$ is
{\em{a linear torus curve of type $(p,q)$}}
if  $F_q=L^q$ for some linear form $L$.
If $C$ is a linear torus curve of type $(p,q)$,
then $C$ generically consists of $q$ components of curves of degree $p$
as
 \[
F=\prod_{j=1}^{q}(F_p-\zeta^jL^p),\quad {\text{where}}\quad
\zeta:=\exp \left(\frac{2\pi \sqrt{-1}}{q}\right)
      \]
      and the inner singularities of $C$  are contained in
the intersection  $C_p \cap L$.
If $C_p\cap L$ consists of $p$ distinct points,
 $C$ is called {\em a generic linear torus curve}.
 In this case, $C$ has $p$ singular points, each of which is
topologically equivalent to $B_{pq,q}$.  

If $C$ is a generic linear torus curve of type $(p,q)$,
then $C$ has $q$ components of degree $p$
which are defined by
$C^{(j)}:\{F_p-\zeta^jL^p=0\}$ for $j=1,\dots,q$.
Then $C^{(i)}$ and $C^{(j)}$ 
 are tangent at $p$ points
$\{F_p=L=0\}$
with the intersection multiplicity $p$
each other.
The fundamental group is given by
\[
 \pi_1(\Bbb{P}^2\setminus C)\cong
 F(q-1)* \Bbb{Z}/p\Bbb{Z}
\]
where $F(d)$ is the free group of rank $d$ (\cite{Okacertain}).
Its  Alexander polynomial $\Delta_C(t)$ is given by
 \[
   \Delta_C(t)=
   \dfrac{(t^{pq}-1)^{q-1}(t-1)}{t^q-1}.
 \]
 
 Let $C$ be a tame $(p,q)$ linear torus curve of a maximal contact.
Then $C$ has $q$ components of degree $p$
which intersect at  one point with intersection multiplicity $p^2$ each other.
  In this paper,
we compute that fundamental groups of $\Bbb{P}^2\setminus C$ and
$\Bbb{C}^2\setminus C$.

Our main result is the following:
 \begin{theorem}
   Let $C$ be a tame $(p,q)$ linear torus curve of a maximal contact.
  Then the fundamental group $\pi(\Bbb{P}^2\setminus C)$
  is isomorphic to  that of generic linear torus curves.
  Namely
 \[\begin{split}
& \pi(\Bbb{C}^2\setminus C)\cong
\langle g_1,\dots,g_q,\omega \,|\,\omega=g_1\cdots g_q,[g_j,\omega^p]=e,j=1,\dots,q    \rangle\\
& \pi(\Bbb{P}^2\setminus C)\cong
\langle g_1,\dots,g_q,\omega \,|\,\omega^p=e,\
    \omega=g_1\cdots g_q \rangle
 \cong  F(q-1)* \Bbb{Z}/p\Bbb{Z}.
\end{split}
\]
 \end{theorem}

\section{Preliminaries}
\subsection{Van Kampen-Zariski Pencil method}
Let $C$ be a reduced plane curve of degree $d$ in $\Bbb{P}^2$.
To compute the fundamental groups
$\pi_1(\Bbb{P}^2\setminus C)$ and $\pi_1(\Bbb{C}^2\setminus C)$,
we use so-called {\em{Van Kampen-Zariski pencil method}}.
We recall it briefly in  the following (\cite{OkaSurvey}).
We fix a point $B_{0}\in \Bbb{P}^2\setminus C$ and
we consider the set of lines
$\mathcal{L}=\{L_{s}\,|\,s\in \Bbb{P}^1\}$ through $B_0$ and
$\mathcal{L}$ is called a {\em pencil}. 
Taking a linear change of coordinates if necessary,
we may assume that $B_{0}=[1:0:0]$ and $L_{s}$ is defined by 
$L_{s}=\{Y-s Z=0\}$ in $\Bbb{P}^2$ where $(X,Y,Z)$ is the fixed  
homogeneous coordinates.
Take $L_{\infty}=\{Z=0\}$ as the line at infinity and 
assume that $L_{\infty}$ intersects transversely with $C$.
We consider the affine coordinates $(x,y)=(X/Z,Y/Z)$ on
$\Bbb{C}^2=\Bbb{P}^2-L_{\infty}$.
Let $F(X,Y,Z)$ be the defining homogeneous polynomial of $C$ and 
let $f(x,y)=F(x,y,1)$ be the affine equation of $C$.
We use the following notations:
\[
  C^a=C\cap \Bbb{C}^2,\quad L^a_{s}=L_{s}\cap \Bbb{C}^2.
\] 
We identify $L_{s}$ and $L^a_{s}$ with
$\Bbb{P}^1$ and $\Bbb{C}$ respectively and
the pencil line  $L_{s}^a$ is defined by 
$\{y=s\}$ in the affine coordinates $(x,y)$.
We use $x$ as the coordinates of $L_{s}^a$.

A pencil line $L_{s}$ is called
{\em{singular with respect to $C$}}
if $L_{s}$ passes through a singular point of $C$ or
   $L_{s}$ is tangent to $C$.
Otherwise, we call $L_{s}$ is {\em{generic}}. 
Hearafter we assume that $L_{\infty}$ is generic and
          $B_{0}$ is not contained in $C$.

Let $\Bbb{C}_y$ be the space of
the parameters of the pencil with
coordinates $y$ and
let $\Sigma=\{ s\in \Bbb{C}_y\,|\,
{\text{$L_{s}$ is a singular pencil lines}}\}$
and suppose that 
$\Sigma=\{s_1,\dots, s_k \}\subset  \Bbb{C}_y$. 
We fix a generic pencil line $L_{s_0}$ (so $s_0\in \Bbb{C}_y\setminus
\Sigma$)
 and put 
$L^a_{s_0}\cap C^a=\{Q_1,\dots,Q_{d}\}$
where $d$ is the degree of $C$.
We take a base point $*_0\in L^a_{s_0}\setminus L^a_{s_0}\cap C^a$ 
on the real axis which is sufficiently near to 
$B_0$ and $*_0\ne B_0$.
We take a large disk $\Delta_{R}\subset L^a_{s_0}$
such that $L^a_{s_0}\cap C^a \subset \Delta_{R}$ and
$*_{0}\not \in \Delta_{R}$.
We may assume that
$\Delta_{R}=\{(x,s_0)\in L_{s_0}^a\,|\,|x|\le R\}$ with a
sufficient  large $R$.
We orient the boundary of $\Delta_{R}$
counter-clockwise and
we put $\Xi=\partial\Delta_{R}$.  
Join the circle $\Xi$ to the base point by a line segment $L$ connecting $*_0$ and 
     $\Xi$ along the real axis.
Let $\Omega$ be the class of this loop 
$L\circ \Xi \circ L^{-1}$ in
$\pi_1(L^a_{s_0}\setminus L^a_{s_0}\cap C;*_0)$.
We take free generators $g_1,\dots,g_d$ of 
$\pi_1(L^a_{s_0}\setminus L^a_{s_0}\cap C;*_0)$ so that 
$g_i$ goes around $Q_i$ counter-clockwise along a small circle and 
 we assume that
$\omega=g_d\cdots g_1$, taking a suitable ordering
of $g_1,\dots, g_d$ if necessary.

\begin{center}
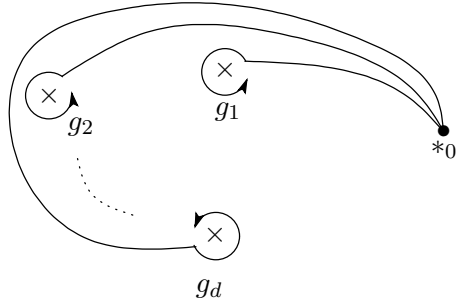
\begin{figure}[H]
\centering
\unitlength 0.1in
\begin{picture}( 24.2600, 14.1200)( 21.1000,-25.2500)
\put(33.9200,-14.6700){\makebox(0,0){$\times$}}%
\put(24.7000,-16.0000){\makebox(0,0){$\times$}}%
\put(33.4000,-23.3000){\makebox(0,0){$\times$}}%
%
\special{pn 8}%
\special{ar 3392 1476 122 122  0.3652014 5.8010213}%
%
\special{pn 8}%
\special{pa 3500 1530}%
\special{pa 3506 1518}%
\special{fp}%
\special{sh 1}%
\special{pa 3506 1518}%
\special{pa 3462 1572}%
\special{pa 3484 1566}%
\special{pa 3498 1586}%
\special{pa 3506 1518}%
\special{fp}%
\put(45.3600,-17.8700){\makebox(0,0){$\bullet$}}%
%
\special{pn 8}%
\special{pa 4536 1788}%
\special{pa 4514 1764}%
\special{pa 4492 1740}%
\special{pa 4470 1716}%
\special{pa 4448 1692}%
\special{pa 4426 1670}%
\special{pa 4402 1650}%
\special{pa 4378 1628}%
\special{pa 4354 1610}%
\special{pa 4328 1592}%
\special{pa 4302 1576}%
\special{pa 4276 1560}%
\special{pa 4248 1546}%
\special{pa 4220 1532}%
\special{pa 4192 1520}%
\special{pa 4162 1510}%
\special{pa 4134 1498}%
\special{pa 4104 1490}%
\special{pa 4072 1482}%
\special{pa 4042 1474}%
\special{pa 4010 1466}%
\special{pa 3978 1460}%
\special{pa 3946 1454}%
\special{pa 3914 1450}%
\special{pa 3882 1444}%
\special{pa 3848 1440}%
\special{pa 3816 1438}%
\special{pa 3782 1434}%
\special{pa 3748 1432}%
\special{pa 3714 1430}%
\special{pa 3680 1428}%
\special{pa 3646 1426}%
\special{pa 3612 1424}%
\special{pa 3576 1422}%
\special{pa 3542 1422}%
\special{pa 3508 1420}%
\special{pa 3496 1420}%
\special{sp}%
%
\special{pn 8}%
\special{ar 2470 1600 122 122  6.1954987 6.2831853}%
\special{ar 2470 1600 122 122  0.0000000 5.3393627}%
%
\special{pn 8}%
\special{pa 2592 1602}%
\special{pa 2590 1588}%
\special{fp}%
\special{sh 1}%
\special{pa 2590 1588}%
\special{pa 2580 1658}%
\special{pa 2598 1642}%
\special{pa 2620 1652}%
\special{pa 2590 1588}%
\special{fp}%
%
\special{pn 8}%
\special{pa 4536 1788}%
\special{pa 4520 1760}%
\special{pa 4504 1732}%
\special{pa 4486 1704}%
\special{pa 4468 1678}%
\special{pa 4450 1652}%
\special{pa 4430 1628}%
\special{pa 4410 1604}%
\special{pa 4388 1580}%
\special{pa 4366 1558}%
\special{pa 4342 1538}%
\special{pa 4318 1518}%
\special{pa 4294 1500}%
\special{pa 4266 1482}%
\special{pa 4240 1466}%
\special{pa 4212 1450}%
\special{pa 4184 1436}%
\special{pa 4154 1422}%
\special{pa 4126 1408}%
\special{pa 4094 1394}%
\special{pa 4064 1382}%
\special{pa 4034 1372}%
\special{pa 4002 1360}%
\special{pa 3970 1350}%
\special{pa 3938 1340}%
\special{pa 3904 1330}%
\special{pa 3872 1322}%
\special{pa 3840 1314}%
\special{pa 3806 1304}%
\special{pa 3774 1296}%
\special{pa 3740 1290}%
\special{pa 3708 1282}%
\special{pa 3674 1274}%
\special{pa 3642 1268}%
\special{pa 3608 1262}%
\special{pa 3576 1256}%
\special{pa 3544 1250}%
\special{pa 3510 1246}%
\special{pa 3478 1240}%
\special{pa 3446 1236}%
\special{pa 3412 1234}%
\special{pa 3380 1230}%
\special{pa 3348 1228}%
\special{pa 3316 1228}%
\special{pa 3284 1228}%
\special{pa 3252 1228}%
\special{pa 3220 1228}%
\special{pa 3190 1230}%
\special{pa 3158 1234}%
\special{pa 3128 1236}%
\special{pa 3096 1242}%
\special{pa 3066 1248}%
\special{pa 3036 1254}%
\special{pa 3006 1262}%
\special{pa 2976 1270}%
\special{pa 2946 1280}%
\special{pa 2916 1292}%
\special{pa 2888 1302}%
\special{pa 2858 1316}%
\special{pa 2830 1328}%
\special{pa 2802 1344}%
\special{pa 2772 1358}%
\special{pa 2744 1374}%
\special{pa 2716 1390}%
\special{pa 2688 1406}%
\special{pa 2660 1422}%
\special{pa 2632 1440}%
\special{pa 2604 1456}%
\special{pa 2576 1474}%
\special{pa 2550 1492}%
\special{pa 2540 1498}%
\special{sp}%
%
\special{pn 8}%
\special{ar 3344 2336 124 124  3.5928952 6.2831853}%
\special{ar 3344 2336 124 124  0.0000000 2.7399094}%
%
\special{pn 8}%
\special{pa 3240 2274}%
\special{pa 3236 2286}%
\special{fp}%
\special{sh 1}%
\special{pa 3236 2286}%
\special{pa 3278 2230}%
\special{pa 3254 2236}%
\special{pa 3240 2216}%
\special{pa 3236 2286}%
\special{fp}%
%
\special{pn 8}%
\special{pa 4536 1788}%
\special{pa 4532 1754}%
\special{pa 4528 1722}%
\special{pa 4524 1690}%
\special{pa 4516 1660}%
\special{pa 4508 1630}%
\special{pa 4496 1600}%
\special{pa 4482 1574}%
\special{pa 4466 1548}%
\special{pa 4448 1524}%
\special{pa 4428 1500}%
\special{pa 4406 1478}%
\special{pa 4380 1458}%
\special{pa 4356 1438}%
\special{pa 4328 1418}%
\special{pa 4300 1400}%
\special{pa 4270 1384}%
\special{pa 4240 1368}%
\special{pa 4208 1352}%
\special{pa 4176 1338}%
\special{pa 4144 1324}%
\special{pa 4112 1310}%
\special{pa 4080 1296}%
\special{pa 4048 1284}%
\special{pa 4016 1272}%
\special{pa 3984 1260}%
\special{pa 3952 1248}%
\special{pa 3920 1236}%
\special{pa 3888 1226}%
\special{pa 3858 1216}%
\special{pa 3826 1206}%
\special{pa 3794 1198}%
\special{pa 3762 1188}%
\special{pa 3732 1180}%
\special{pa 3700 1172}%
\special{pa 3668 1166}%
\special{pa 3638 1158}%
\special{pa 3606 1152}%
\special{pa 3574 1146}%
\special{pa 3544 1140}%
\special{pa 3512 1136}%
\special{pa 3480 1132}%
\special{pa 3450 1128}%
\special{pa 3418 1124}%
\special{pa 3386 1122}%
\special{pa 3354 1118}%
\special{pa 3324 1116}%
\special{pa 3292 1116}%
\special{pa 3260 1114}%
\special{pa 3228 1114}%
\special{pa 3196 1114}%
\special{pa 3164 1114}%
\special{pa 3134 1116}%
\special{pa 3102 1118}%
\special{pa 3070 1120}%
\special{pa 3036 1122}%
\special{pa 3004 1126}%
\special{pa 2972 1128}%
\special{pa 2940 1134}%
\special{pa 2908 1138}%
\special{pa 2874 1144}%
\special{pa 2842 1150}%
\special{pa 2808 1156}%
\special{pa 2776 1162}%
\special{pa 2742 1170}%
\special{pa 2710 1178}%
\special{pa 2676 1188}%
\special{pa 2644 1198}%
\special{pa 2612 1208}%
\special{pa 2580 1220}%
\special{pa 2550 1232}%
\special{pa 2520 1246}%
\special{pa 2492 1262}%
\special{pa 2466 1278}%
\special{pa 2440 1296}%
\special{pa 2414 1316}%
\special{pa 2392 1336}%
\special{pa 2372 1358}%
\special{pa 2352 1382}%
\special{pa 2334 1408}%
\special{pa 2320 1436}%
\special{pa 2306 1464}%
\special{pa 2294 1494}%
\special{pa 2284 1524}%
\special{pa 2276 1556}%
\special{pa 2270 1588}%
\special{pa 2266 1622}%
\special{pa 2264 1656}%
\special{pa 2262 1690}%
\special{pa 2264 1724}%
\special{pa 2266 1758}%
\special{pa 2270 1792}%
\special{pa 2276 1826}%
\special{pa 2282 1860}%
\special{pa 2292 1894}%
\special{pa 2302 1926}%
\special{pa 2312 1958}%
\special{pa 2326 1990}%
\special{pa 2340 2022}%
\special{pa 2356 2052}%
\special{pa 2372 2080}%
\special{pa 2390 2110}%
\special{pa 2410 2136}%
\special{pa 2430 2162}%
\special{pa 2450 2188}%
\special{pa 2472 2212}%
\special{pa 2496 2236}%
\special{pa 2520 2256}%
\special{pa 2546 2276}%
\special{pa 2570 2296}%
\special{pa 2598 2312}%
\special{pa 2624 2328}%
\special{pa 2652 2342}%
\special{pa 2680 2354}%
\special{pa 2710 2366}%
\special{pa 2740 2374}%
\special{pa 2770 2382}%
\special{pa 2800 2388}%
\special{pa 2832 2394}%
\special{pa 2862 2398}%
\special{pa 2894 2400}%
\special{pa 2926 2402}%
\special{pa 2958 2404}%
\special{pa 2992 2404}%
\special{pa 3024 2404}%
\special{pa 3058 2404}%
\special{pa 3092 2402}%
\special{pa 3124 2400}%
\special{pa 3158 2398}%
\special{pa 3192 2396}%
\special{pa 3226 2392}%
\special{pa 3240 2392}%
\special{sp}%
\put(34.0800,-16.8300){\makebox(0,0){$g_1$}}%
\put(26.4000,-17.6000){\makebox(0,0){$g_2$}}%
\put(33.1000,-26.1000){\makebox(0,0){$g_d$}}%
\put(45.4000,-18.8000){\makebox(0,0){$*_0$}}%
%
\special{pn 8}%
\special{pa 2620 1930}%
\special{pa 2628 1964}%
\special{pa 2636 1996}%
\special{pa 2646 2028}%
\special{pa 2658 2056}%
\special{pa 2672 2084}%
\special{pa 2688 2110}%
\special{pa 2708 2130}%
\special{pa 2732 2150}%
\special{pa 2758 2166}%
\special{pa 2786 2182}%
\special{pa 2816 2194}%
\special{pa 2848 2206}%
\special{pa 2882 2216}%
\special{pa 2914 2228}%
\special{pa 2944 2236}%
\special{sp -0.045}%
\end{picture}%
\vspace{0.4cm}
 \caption{$L_{s_0}\cap C$}
\end{figure}
\end{center}
Hereafter
we denote a small lasso oriented in the counter clockwise
direction by a bullet with a path in the following figures.
Thus
${\vrule height3pt depth-2.3pt width1cm}\hspace{-0.09cm} \bullet$ 
indicates
 ${\vrule height3pt depth-2.3pt width1cm}\hspace{-0.11cm}  \circlearrowleft$.

The fundamental group $\pi_1(\Bbb{C}_{y}\setminus \Sigma;s_0)$ acts on 
$\pi_1(L^a_{s_0}\setminus L^a_{s_0}\cap C;*_0)$ which we call  
{\em{the monodromy action}} of $\pi_1(\Bbb{C}_{y}\setminus
\Sigma;s_0)$. For details, we refer to  \cite{OkaSurvey,OkaTwo}.
Note that
$\pi_1(L^a_{s_0}\setminus L^a_{s_0}\cap C;*_0)$
is a free group of rank $d$ with generators $g_1,\dots,g_d$.
The class of the action
of $\sigma\in \pi_1(\Bbb{C}_y\setminus \Sigma;s_0)$
on $g\in \pi_1(L^a_{s_0}\setminus L^a_{s_0}\cap C;*_0)$ is
denoted by $g^{\sigma}$. 



Let $\mathcal{M}$ be the normal subgroup of
$\pi_1(L^a_{s_0}\setminus L^a_{s_0}\cap C;*_0)$
which is normally generated  by 
\[\mathcal{R}=\langle g^{-1}g^{\sigma}\,|\,g\in
\pi_1(L^a_{s_0}\setminus L^a_{s_0}\cap C;*_0),\ \sigma
\in \pi_1(\Bbb{C}_y\setminus \Sigma;s_0)\rangle\]
and we call $\mathcal{M}$ the {\em{the group of the monodromy relations}}. 
Put
 \[
\mathcal{M}(\sigma_i)
=\{g_j^{-1}g_j^{\sigma_i}\,|\,j=1,\dots,d\}.
\]
Then the group $\mathcal{M}$
is the minimal normal subgroup of
$\pi_1(L^a_{s_0}\setminus L^a_{s_0}\cap C;*_0)$
generated by $\bigcup_{j=1}^{k}\mathcal{M}(\sigma_i)$.
By the definition, we have the relation
\[
 \tag*{$R(\sigma_i)$} g_{j}=g_{j}^{\sigma_i}
\]
in the quotient group
$\pi_1(L^a_{s_0}\setminus L^a_{s_0}\cap C;*_0)/\mathcal{M}$.
We call $R(\sigma_i)$ {\em{the monodromy relation for $\sigma_i$}}.
Let $j:L^a_{s_0}\setminus L^a_{s_0}\cap C\to \Bbb{C}^2\setminus C^a$  and 
$\iota:\Bbb{C}^2\setminus C^a\to \Bbb{P}^2\setminus C$
be the respective inclusions.
\begin{proposition}[\cite{Za1,Za4,Kampen1}]
Under the above situations,
\begin{enumerate}
\item The canonical homomorphism
$j_{\sharp}:\pi_1(L^a_{s_0}\setminus L^a_{s_0}\cap C;*_0)\to
\pi_1(\Bbb{C}^2\setminus C^a;*_0)$ is surjective and the kernel
${\rm{Ker}}j_{\sharp}$ is equal to $\mathcal{M}$.
      Thus we have the isomorphism:
\[
  \begin{split}
 &\pi_1(\Bbb{C}^2\setminus C^a;*_0)\cong
 \pi_1(L^a_{s_0}\setminus L^a_{s_0}\cap C;*_0)
                                             /\mathcal{M}.
  \end{split}
\]
\item $($\cite{OkaCentral}$)$
     The canonical homomorphism
$\iota_{\sharp}:\pi_1(\Bbb{C}^2\setminus C^a;*_0)\to
      \pi_1(\Bbb{P}^2\setminus C;*_0)$ is surjective and 
the kernel  ${\rm{Ker}}\iota_{\sharp}$ is  generated by a single element
      $\omega=g_d\cdots g_1$ which is in the center of
      $\pi(\Bbb{C}^2\setminus C^a)$
and ${\rm{Ker}}\iota_{\sharp}=\langle \omega \rangle\cong \Bbb{Z}$.
      Thus we have an isomorphism
      \[
   \pi_1(\Bbb{P}^2\setminus C;*_0)\cong
      \pi_1(\Bbb{C}^2\setminus C^{a};*_0)/\langle \omega \rangle
      \]
\end{enumerate}
\end{proposition}

 \section{Proof of Theorem 1}
Let $(x,y)$ be affine coordinates such that $x=X/Z,\,y=Y/Z$ 
on $\Bbb{C}^2:=\Bbb{P}^2\setminus\{Z=0\}$.

\subsection{Construction of curves} 
In this section,
we construct
a linear torus curve $C$ of a maximal contact and
investigate its local properties.
First we introduce a  plane curve
$D_{\alpha}=\{g_\alpha(x,y)=0\}$ of degree $p$ where
the defining polynomial $g_{\alpha}(x,y)$ is defined by 
\[
 g_{\alpha}(u,y)=u-\psi(y,\alpha),\quad 
 \psi(y,\alpha)=y-\alpha y^p,\quad
 \alpha\in \Bbb{C}^{*}.
\]
Now we consider the $p$-fold cyclic covering (\cite{OkaTwo})
which is defined by
\[
 \varphi_p:\Bbb{C}^2\to \Bbb{C}^2,\quad
 \varphi_p(x,y)=(u,y),\quad u=x^p.
\]
To distinguish two affine planes, we denote the 
source space of $\varphi_p$ by $\Bbb{C}^2_s$ with coordinates $(x,y)$
and the target space of   $\varphi_p$ by $\Bbb{C}^2_t$ 
with coordinates $(u,y)$.   
Hereafter we simply denote $\Bbb{C}^2$ instead of $\Bbb{C}^2_s$.

Let $C_{\alpha}:=\varphi^{-1}(D_{\alpha})$
be the pull-back of $D_\alpha$ by $\varphi_p$
and let
$f_{\alpha}(x,y)=g_{\alpha}(x^p,y)$ be the defining polynomial of
$C_{\alpha}$.
Note that 
\[
 f_{\alpha}(x,y)=x^p-\psi(y,\alpha).
\]
By the defining equation of $C_\alpha$,
we see that the set of parameters which corresponds to
the singular pencil lines
for $C_\alpha$ are given by
 \[
\Sigma_{\alpha}:=  \{y\in \Bbb{C}_y\,|\,\psi(y,\alpha)=0\}
 \]
(cf. \cite{Okacertain}).
Fix a $\gamma$ such that $\gamma^{p-1}=1/\alpha$.
Then we factorize $\psi(y,\alpha)$ as follows:
\[
 \psi(y,\alpha)=\alpha\, y\prod_{k=0}^{p-2}(\gamma \xi^{k}-y),\quad
  \xi:=\exp\left(\frac{2\pi \sqrt{-1}}{p-1}\right).
\]
Then we can see that
$O=(0,0)$ and $Q_{k}=(0,\gamma\, \xi^{k})$ for $k=0,\dots,p-2$
are flex points of $C_{\alpha}$ of flex order $p-2$ 
and
their tangent lines are nothing but
the singular pencil lines through
these points and
they are given by
 $y=0$ and $y=\gamma\xi^{k}$ respectively.

Now we are ready to define a reduced curve $C$.
Take $q$ non-zero mutually distinct complex numbers
$\alpha_1,\dots,\alpha_q$ and 
put $D_j=\{g_{\alpha_j}(x,y)=0\}$ for $j=1,\dots, q$ and put 
    $D=\bigcup_{i=1}^q D_j $.
Then put $C_j=\varphi_p^{-1}(D_j)$ for $=1,\dots,q$ and finally 
 we define 
\[
  C=\varphi_p^{-1}(D)=C_1\cup\cdots \cup C_q.
\]
The defining polynomials $f_j(x,y)$ and $f(x,y)$ of $C_j$ and $C$
respectively are given as follows.
\[
 f_j(x,y)=x^p-\psi(y,\alpha_j),\quad
 f(x,y)=\prod_{j=1}^qf_j(x,y).
\]
Put $ U=\{(\alpha_1,\dots,\alpha_q)\in \Bbb{C^{*}}^q\,|
         \,\alpha_i\ne \alpha_j, {\text{ for any $i\ne j$}}\}$.
It is known that the embedded topology of $C\subset \Bbb{C}^2$
 does not depend on the choice of
$(\alpha_1,\dots,\alpha_q)\in U$ (see \cite{BenoitTuCorrect}).

\begin{lemma}
The reduced curve $C$ can be a linear torus curve of a maximal contact
 for a certain  choice of $(\alpha_1,\dots,\alpha_q)$.
\end{lemma}
\begin{proof}
We take
$(\alpha_1,\dots,\alpha_q)=(1,\zeta,\dots,\zeta^{p-1})\in U$,
 then we claim that $C=\{f(x,y)=0\}$ is
 a $(p,q)$-linear torus curve of a maximal contact. 
Indeed, $f(x,y)$ takes  the form: 
\[
\begin{split}
 f(x,y)&=\prod_{j=1}^q
 \left(
  \zeta^{j-1}y^p-y+x^p
 \right)\\&=
 (y^p)^{q}-(y-x^p)^q\\&= (y^q)^{p}-(y-x^p)^q.
\end{split}
\]
This expression  shows that  $C$ is a $(p,q)$-linear torus curve of a maximal contact.
\end{proof}
For  practical computations, 
we suppose hereafter that
$\alpha_1,\dots,\alpha_q$
are real numbers such that
$\alpha_1 > \dots > \alpha_q>0$.
Let $\gamma_j$ be a real positive number such that
$\gamma_j^{p-1}=1/\alpha_j$ for $j=1,\dots,q$.
By the assumption $\alpha_1>\dots>\alpha_q>0$, we have 
\[
 0<\gamma_1< \cdots < \gamma_q.
\]
As $C_j\cap C_i=\{O\}$ for any $j\ne i$,
the possible singular pencil $L_{s}=\{y=s\}$
is either $\{y=0\}$ or $L_{s}$ is tangent to one of $C_j$
outside of $O$.
 
\begin{lemma}\label{local-data}
 Under the above situation,
 the local data of $C$ for the calculation of
 the fundamental group
 of $\Bbb{P}^2\setminus C$
 is the following.
 \begin{enumerate}
  \item Singular pencil lines are $y=0$ and $y=\gamma_j\xi^{k}$
          for $j=1,\dots,q$ and $k=0,\dots, p-2$.
        The pencil lines $y=\gamma_j\xi^{k}$ is tangent to $C_j$ at
        $Q_{j,k}:=(0,\gamma_j\xi^{k})$.
  \item  Two curves $C_j$ and $C_i$ $(j\ne i)$ intersect only at
         $O\in \Bbb{C}^2$
         and
         $I(C_j,C_i;O)=p^2$.
   \item  The singularity type $C$ at $O$ is given by
         $(C,O)\sim B_{p^2q,q}$.
 \end{enumerate}
\end{lemma}


\subsection{Calculation of the fundamental group
 $\pi_1(\Bbb{P}^2\setminus C)$ and  $\pi_1(\Bbb{C}^2\setminus C)$}
For the calculations of the fundamental groups
$\pi_1(\Bbb{P}^2\setminus C)$ and  $\pi_1(\Bbb{C}^2\setminus C)$,
 we use the van Kampen-Zariski pencil method.
We take the base point $B_{0}=[1:0:0]$ in $\Bbb{P}^2$ and 
consider the  pencil
$\mathcal{L}=\{L_s\,|\, s\in \Bbb{C}\}$ through $B_{0}$
   with 
$L_{s}=\{Y=s Z\}$.
The line at infinity $L_{\infty}$ is given by $\{Z=0\}$.
 Then $L_{\infty}$ is generic with respect to $C$.
Affine pencil is 
$\mathcal{L}^a=\{L^a_{s}\}_{s\in \Bbb{C}}$
    with $L^a_{s}=\{y=s\}$.
(By abuse of notation, we consider this pencil $\mathcal{L}=\{L_s\,|\, s\in \Bbb{C}\}$ in $\Bbb{C}^2_t$ and $\Bbb{C}^2_s$.)
By Lemma \ref{local-data},
 the set of parameters
$\Sigma \subset \Bbb{C}_y$ which corresponds to
singular pencil lines for $C$ is given as follows:
\[
 \Sigma:=\{0,\gamma_{j}\xi^{k}\in \Bbb{C}_y\,|\,
   k=0,\dots,p-2,\ j=1,\dots,q \}.
\]

\noindent
Take the base point
$\gamma_0$ of $\Bbb{C}_y\setminus \Sigma$
 on the real axis so that $0<\gamma_0<\gamma_1$.
As
\[
 \psi(\gamma_0,\alpha_j)- \psi(\gamma_0,\alpha_i)=
 (\alpha_i-\alpha_j)\gamma_0^p>0\quad
 {\text{ if }} \quad i<j,
\]
we have 
\[
   0<\psi(\gamma_0,\alpha_1)<\psi(\gamma_0,\alpha_2)< \cdots < \psi(\gamma_0,\alpha_q).
 \]
We take the base point $*_0=(\tau_0,\gamma_0)$ where $\tau_0$ is a
sufficiently large positive number. 
As $C$ is the pull-back of $D$ by the $p$-fold cyclic covering
$\varphi_p:(x,y)\mapsto (x^p,y)$,  the monodromy relations for
$\pi_1(L^a_{\gamma_0}\setminus L^a_{\gamma_0}\cap C;*_0)$
are essentially obtained by taking lifting 
 the monodromy relations 
 for $\pi_1(L^a_{\gamma_0}\setminus ( L^a_{\gamma_0}\cap D)\cup
 \{0\};*_t)$ by $\varphi_p$ where the  base point $*_t$ is a  real point 
 defined by $*_t=(\tau_0^p,\gamma_0)$  
 (\cite{OkaTwo}).
This is the basic idea for the computation of the fundamental groups.

We first take loops  $b_1,\dots, b_q$ of
 $\pi_1(L^a_{\gamma_0}\setminus (L^a_{\gamma_0}\cap D)\cup \{0\},*_t)$
 and put $\tau:=b_1\cdots b_q$
 as in Figure \ref{g-3}.
 \begin{center}
\begin{figure}[H]
\centering
\unitlength 0.1in
\begin{picture}( 28.9700,  7.4000)(  2.4300,-11.2400)
%
\special{pn 8}%
\special{pa 410 822}%
\special{pa 1460 822}%
\special{fp}%
%
\special{pn 8}%
\special{pa 1880 822}%
\special{pa 3140 822}%
\special{fp}%
%
\special{pn 8}%
\special{pa 1250 822}%
\special{pa 2300 822}%
\special{dt 0.045}%
%
\special{pn 8}%
\special{pa 1796 548}%
\special{pa 1800 582}%
\special{pa 1806 614}%
\special{pa 1812 646}%
\special{pa 1824 674}%
\special{pa 1838 700}%
\special{pa 1858 724}%
\special{pa 1880 744}%
\special{pa 1908 762}%
\special{pa 1936 778}%
\special{pa 1966 794}%
\special{pa 1998 808}%
\special{pa 2028 822}%
\special{sp}%
%
\special{pn 8}%
\special{pa 1796 548}%
\special{pa 1764 554}%
\special{pa 1734 562}%
\special{pa 1702 568}%
\special{pa 1670 574}%
\special{pa 1638 582}%
\special{pa 1608 588}%
\special{pa 1576 596}%
\special{pa 1546 604}%
\special{pa 1514 614}%
\special{pa 1484 622}%
\special{pa 1454 632}%
\special{pa 1424 644}%
\special{pa 1394 654}%
\special{pa 1366 668}%
\special{pa 1336 680}%
\special{pa 1308 694}%
\special{pa 1278 708}%
\special{pa 1250 724}%
\special{pa 1222 740}%
\special{pa 1194 754}%
\special{pa 1166 772}%
\special{pa 1140 788}%
\special{pa 1112 804}%
\special{pa 1084 820}%
\special{pa 1082 822}%
\special{sp}%
%
\special{pn 8}%
\special{pa 1796 548}%
\special{pa 1768 564}%
\special{pa 1740 578}%
\special{pa 1712 594}%
\special{pa 1684 608}%
\special{pa 1656 624}%
\special{pa 1628 640}%
\special{pa 1600 656}%
\special{pa 1572 672}%
\special{pa 1546 688}%
\special{pa 1518 706}%
\special{pa 1492 724}%
\special{pa 1464 740}%
\special{pa 1438 758}%
\special{pa 1412 776}%
\special{pa 1384 794}%
\special{pa 1358 812}%
\special{pa 1344 822}%
\special{sp}%
\put(10.9800,-8.3700){\makebox(0,0){{\large{$\bullet$}}}}%
\put(13.5000,-8.3700){\makebox(0,0){{\large{$\bullet$}}}}%
\put(20.3200,-8.3700){\makebox(0,0){{\large{$\bullet$}}}}%
%
\special{pn 8}%
\special{pa 1796 548}%
\special{pa 1830 548}%
\special{pa 1862 546}%
\special{pa 1896 546}%
\special{pa 1928 544}%
\special{pa 1960 546}%
\special{pa 1994 546}%
\special{pa 2026 550}%
\special{pa 2056 552}%
\special{pa 2088 558}%
\special{pa 2118 566}%
\special{pa 2148 574}%
\special{pa 2178 584}%
\special{pa 2206 598}%
\special{pa 2236 612}%
\special{pa 2264 628}%
\special{pa 2292 644}%
\special{pa 2320 660}%
\special{pa 2348 676}%
\special{pa 2376 690}%
\special{pa 2404 706}%
\special{pa 2434 718}%
\special{pa 2462 730}%
\special{pa 2492 742}%
\special{pa 2522 752}%
\special{pa 2552 760}%
\special{pa 2584 768}%
\special{pa 2614 776}%
\special{pa 2646 782}%
\special{pa 2678 788}%
\special{pa 2708 794}%
\special{pa 2740 798}%
\special{pa 2772 802}%
\special{pa 2806 806}%
\special{pa 2838 808}%
\special{pa 2870 810}%
\special{pa 2902 814}%
\special{pa 2936 816}%
\special{pa 2968 816}%
\special{pa 3002 818}%
\special{pa 3034 820}%
\special{pa 3066 822}%
\special{sp}%
\put(30.7700,-9.4700){\makebox(0,0){$*_t$}}%
\put(11.7600,-6.3200){\makebox(0,0){{\small{$b_1$}}}}%
\put(16.3000,-7.4000){\makebox(0,0){{\small{$b_2$}}}}%
\put(20.1600,-7.0500){\makebox(0,0){{\small{$b_q$}}}}%
%
\special{pn 8}%
\special{pa 3056 822}%
\special{pa 3026 814}%
\special{pa 2994 808}%
\special{pa 2962 800}%
\special{pa 2932 792}%
\special{pa 2900 784}%
\special{pa 2870 776}%
\special{pa 2838 766}%
\special{pa 2808 758}%
\special{pa 2778 748}%
\special{pa 2746 738}%
\special{pa 2716 728}%
\special{pa 2686 716}%
\special{pa 2656 706}%
\special{pa 2626 694}%
\special{pa 2596 684}%
\special{pa 2566 672}%
\special{pa 2536 660}%
\special{pa 2506 650}%
\special{pa 2478 638}%
\special{pa 2448 626}%
\special{pa 2420 614}%
\special{pa 2392 604}%
\special{pa 2362 592}%
\special{pa 2334 582}%
\special{pa 2306 570}%
\special{pa 2276 558}%
\special{pa 2248 548}%
\special{pa 2218 538}%
\special{pa 2190 526}%
\special{pa 2160 516}%
\special{pa 2130 506}%
\special{pa 2100 496}%
\special{pa 2068 486}%
\special{pa 2036 478}%
\special{pa 2006 468}%
\special{pa 1972 460}%
\special{pa 1940 452}%
\special{pa 1906 442}%
\special{pa 1872 436}%
\special{pa 1836 428}%
\special{pa 1800 420}%
\special{pa 1764 414}%
\special{pa 1726 408}%
\special{pa 1690 402}%
\special{pa 1652 398}%
\special{pa 1614 394}%
\special{pa 1576 390}%
\special{pa 1538 388}%
\special{pa 1498 386}%
\special{pa 1460 386}%
\special{pa 1424 384}%
\special{pa 1386 386}%
\special{pa 1348 388}%
\special{pa 1312 390}%
\special{pa 1276 394}%
\special{pa 1242 400}%
\special{pa 1206 406}%
\special{pa 1174 414}%
\special{pa 1140 424}%
\special{pa 1110 434}%
\special{pa 1080 446}%
\special{pa 1050 460}%
\special{pa 1024 474}%
\special{pa 998 490}%
\special{pa 974 508}%
\special{pa 950 528}%
\special{pa 930 550}%
\special{pa 910 572}%
\special{pa 894 598}%
\special{pa 878 624}%
\special{pa 866 654}%
\special{pa 856 684}%
\special{pa 848 716}%
\special{pa 842 750}%
\special{pa 838 784}%
\special{pa 838 818}%
\special{pa 840 852}%
\special{pa 844 884}%
\special{pa 850 914}%
\special{pa 860 942}%
\special{pa 872 968}%
\special{pa 888 990}%
\special{pa 904 1010}%
\special{pa 924 1028}%
\special{pa 946 1042}%
\special{pa 972 1056}%
\special{pa 998 1066}%
\special{pa 1026 1074}%
\special{pa 1054 1082}%
\special{pa 1086 1086}%
\special{pa 1118 1090}%
\special{pa 1152 1090}%
\special{pa 1188 1090}%
\special{pa 1224 1090}%
\special{pa 1260 1086}%
\special{pa 1298 1084}%
\special{pa 1336 1078}%
\special{pa 1374 1074}%
\special{pa 1414 1068}%
\special{pa 1452 1060}%
\special{pa 1492 1054}%
\special{pa 1530 1046}%
\special{pa 1570 1038}%
\special{pa 1608 1030}%
\special{pa 1646 1022}%
\special{pa 1684 1014}%
\special{pa 1720 1006}%
\special{pa 1756 1000}%
\special{pa 1790 992}%
\special{pa 1824 986}%
\special{pa 1858 980}%
\special{pa 1892 974}%
\special{pa 1924 968}%
\special{pa 1956 962}%
\special{pa 1986 956}%
\special{pa 2018 952}%
\special{pa 2048 946}%
\special{pa 2078 942}%
\special{pa 2108 938}%
\special{pa 2138 932}%
\special{pa 2168 928}%
\special{pa 2198 924}%
\special{pa 2228 920}%
\special{pa 2258 916}%
\special{pa 2288 912}%
\special{pa 2318 908}%
\special{pa 2348 904}%
\special{pa 2378 900}%
\special{pa 2408 896}%
\special{pa 2440 892}%
\special{pa 2470 888}%
\special{pa 2502 884}%
\special{pa 2534 880}%
\special{pa 2564 876}%
\special{pa 2596 872}%
\special{pa 2628 870}%
\special{pa 2660 866}%
\special{pa 2692 862}%
\special{pa 2724 858}%
\special{pa 2756 854}%
\special{pa 2788 850}%
\special{pa 2820 848}%
\special{pa 2852 844}%
\special{pa 2886 840}%
\special{pa 2918 836}%
\special{pa 2950 834}%
\special{pa 2982 830}%
\special{pa 3016 826}%
\special{pa 3048 822}%
\special{pa 3056 822}%
\special{sp}%
\put(12.1800,-12.0900){\makebox(0,0){$\tau$}}%
\put(6.4100,-9.6800){\makebox(0,0){$0$}}%
\put(17.3300,-7.1600){\makebox(0,0){$\cdots$}}%
\put(30.7000,-8.2000){\makebox(0,0){$\times$}}%
\put(6.4000,-8.2000){\makebox(0,0){$\times$}}%
\end{picture}%
 \vspace{0.5cm}
  \caption{The loops $b_1,\dots, b_q$ in $\{y=\gamma_0\}\cap \Bbb{C}^2_t$}\label{g-3}
\end{figure}
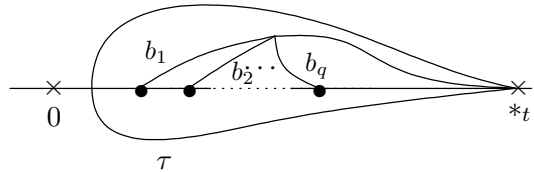
\end{center}
Let $a_{i,j}'$ and $\omega_i'$
be the pull-back of $b_j$  and $\tau$ by $\varphi_p$ respectively
starting from $*_i:=(\eta^i\tau_0,\gamma_0)$ with $i=0,\dots,p-1$ where
$\eta:=\exp(2\pi \sqrt{-1}/p)$
and
let $a_{i,j}$ and $\omega_i$ be the loop
$\ell_i\circ a_{i,j}'\circ  \ell_i^{-1}$ and
$\ell_i\circ \omega_i' \circ \ell_i^{-1}$
where $\ell_i$ is the arc of  the circle $|x|=\tau_0$
from $*_{0}$ to $*_i$ 
as in Figure \ref{g-5}.
Hereafter we identify $a_{i,j}'$ and $a_{i,j}$ in this way. 
\begin{figure}[H]
 \centering
 \input{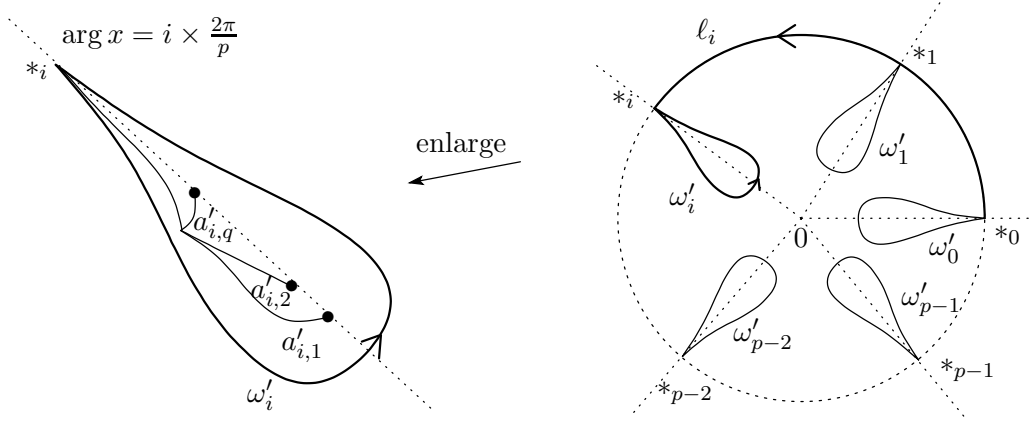}
\caption{The loops $a_{i,j}$ and $\omega_i$ in $\{y=\gamma_0\}\cap \Bbb{C}^2_s$}\label{g-5}
\end{figure}

First we see that
the monodromy relations
on the real axis
in $\Bbb{C}_y$ which 
 correspond to singular pencil lines
 $y=0$ and $y=\gamma_j$ for $j=1,\dots, q$. 
To see these monodromy relations,
we consider following loops $\sigma_0$ and $\sigma_j$ in $\Bbb{C}_y$
 for $j=1,\dots,q$.
First we define the loop $\sigma_0$.
Let $K_0$ be the line segment
from $\gamma_0$ to $0-\varepsilon$ on the real axis and
let $S_0$ be the circle $|y|=\varepsilon$ where the circle is 
always oriented counterclockwise.   
Then  $\sigma_0$ is defined as the loop (see Figure \ref{f-2})
\[
 \sigma_0:=K_0\circ S_0 \circ K_0^{-1}.
\]
Next we define loops  $\sigma_j$ for $j=1,\dots,q$.
Let $S_j$ be the loop which is represented by the circle
$|y-\gamma_j|=\varepsilon$ oriented counter clockwise.
Let $K_j$ be the modified line segment from $\gamma_0$ to $\gamma_j-\varepsilon$.
The segment $[\gamma_i-\varepsilon,\gamma_i+\varepsilon]$ is replaced
by the lower half circle of $S_i$.
Then $\sigma_j$ is defined as the loop (see Figure \ref{f-2})
\[
\sigma_j:=K_j \circ  S_j  \circ  K_j^{-1}.
\]
\begin{center}
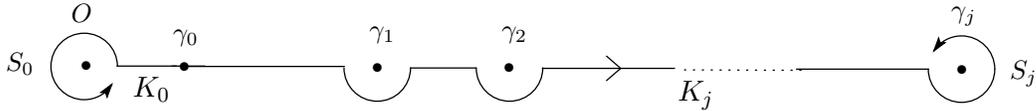
\begin{figure}[H]
\hspace{-2.3cm}
\unitlength 0.1in
\begin{picture}( 58.3600,  5.5200)( 11.7700,-17.7600)
%
\special{pn 8}%
\special{ar 3782 1598 174 174  6.2831853 6.2831853}%
\special{ar 3782 1598 174 174  0.0000000 3.1415927}%
%
\special{pn 8}%
\special{ar 4474 1598 174 174  6.2831853 6.2831853}%
\special{ar 4474 1598 174 174  0.0000000 3.1415927}%
%
\special{pn 8}%
\special{ar 6840 1604 174 174  3.7295953 6.2831853}%
\special{ar 6840 1604 174 174  0.0000000 3.1415927}%
%
\special{pn 8}%
\special{pa 6712 1486}%
\special{pa 6696 1508}%
\special{fp}%
\special{sh 1}%
\special{pa 6696 1508}%
\special{pa 6752 1466}%
\special{pa 6728 1466}%
\special{pa 6720 1442}%
\special{pa 6696 1508}%
\special{fp}%
%
\special{pn 8}%
\special{pa 5976 1604}%
\special{pa 6668 1604}%
\special{fp}%
%
\special{pn 8}%
\special{pa 3956 1598}%
\special{pa 4300 1598}%
\special{fp}%
%
\special{pn 8}%
\special{pa 2780 1590}%
\special{pa 3610 1590}%
\special{fp}%
\put(37.8200,-15.9900){\makebox(0,0){{\scriptsize{$\bullet$}}}}%
\put(44.7300,-15.9900){\makebox(0,0){{\scriptsize{$\bullet$}}}}%
\put(68.4000,-16.0600){\makebox(0,0){{\scriptsize{$\bullet$}}}}%
\put(71.6400,-16.2800){\makebox(0,0){{\small{$S_j$}}}}%
\put(44.9900,-14.2500){\makebox(0,0){{\small{$\gamma_2$}}}}%
\put(27.8000,-14.4000){\makebox(0,0){{\small{$\gamma_{0}$}}}}%
%
\special{pn 8}%
\special{pa 4646 1600}%
\special{pa 5338 1600}%
\special{fp}%
%
\special{pn 8}%
\special{pa 5390 1608}%
\special{pa 6000 1610}%
\special{dt 0.045}%
\put(22.5700,-15.8800){\makebox(0,0){{\scriptsize{$\bullet$}}}}%
\put(22.3500,-13.0900){\makebox(0,0){{\small{$O$}}}}%
%
\special{pn 8}%
\special{ar 2248 1590 174 174  0.5817638 6.2716251}%
%
\special{pn 8}%
\special{pa 2378 1704}%
\special{pa 2392 1684}%
\special{fp}%
\special{sh 1}%
\special{pa 2392 1684}%
\special{pa 2336 1726}%
\special{pa 2360 1728}%
\special{pa 2368 1750}%
\special{pa 2392 1684}%
\special{fp}%
%
\special{pn 8}%
\special{pa 2708 1588}%
\special{pa 2880 1588}%
\special{fp}%
\put(38.0800,-14.2500){\makebox(0,0){{\small{$\gamma_1$}}}}%
\put(68.5100,-13.2500){\makebox(0,0){{\small{$\gamma_j$}}}}%
\put(19.1100,-15.7600){\makebox(0,0){{\small{$S_0$}}}}%
\put(27.7200,-15.9400){\makebox(0,0){{\scriptsize{$\bullet$}}}}%
%
\special{pn 8}%
\special{pa 4982 1536}%
\special{pa 5056 1600}%
\special{fp}%
\special{pa 5056 1600}%
\special{pa 4982 1664}%
\special{fp}%
\put(54.5000,-17.4000){\makebox(0,0){$K_j$}}%
%
\special{pn 8}%
\special{pa 2420 1588}%
\special{pa 2730 1588}%
\special{fp}%
\put(26.0000,-17.1000){\makebox(0,0){$K_0$}}%
\end{picture}%
 \vspace{0.5cm}
  \caption{Loops in $\Bbb{C}_y$}\label{f-2}
\end{figure}
\end{center}

\noindent{\bf{Case 1}}:
First we see the monodromy relations at $y=0$.
By the definitions of $C_j$'s and Lemma \ref{local-data},
the origin $O$ is a flex point of  $C_j$ such that
$\{y=0\}$ is the tangent line
for $j=1,\dots,q$ and 
$C_i$ and $C_j$ intersect with intersection multiplicity $p^2$
at $O$ for each $i\ne j$
and the topological type of $C$ at $O$ is $B_{p^2q,q}$.
\\
To see that monodromy relations,
we look at the Puiseux parametrization of each component $C_{j}$ at $O$.
Consider that curves $D_j$ and $D$ whose
 defining polynomials are $g_j(x,y)=x-\psi(y,\alpha_j)$ and
$g(x,y)=\prod_{j=1}^q g_j(x,y)$ respectively.
By the definitions,
$\psi(y,\alpha_j)=y(y-\alpha_jy^{p-1})$,
$f_j(x,y)=g_j(x^p,y)$,
we have $x^p=y(1-\alpha_jy^{p-1})$.
By the  generalized  binomial  theorem, we can solve
$x^p=y(1-\alpha_jy^{p-1})$  as follows.
\[
\begin{split}
(1)&\quad  C_j:\quad    \begin{cases}
         x=\varphi_j(t),\quad 
           \varphi_j(t)=
          t\left(1 -\dfrac{\alpha_j}{p}\,t^{p\left( p-1\right)}+
                   \cdots\right), \quad j=1,\dots,q.\\
          y=t^p, 
         \end{cases}\\
(2)&\quad            \dfrac{\varphi_j(t)}{t}-\dfrac{\varphi_i(t)}{t}=
           \dfrac{1}{p}(\alpha_i-\alpha_j)t^{p(p-1)}+\cdots,
           \quad  j\ne i.
\end{split}\]
Note that the leading term of $\varphi_j(t)$ is $t$ which is 
independent of index  $j=1,\dots,q$.
The topological behavior of the centers of the  generators,
 $pq$ points $C\cap \{y=\varepsilon \exp( \sqrt{-1} \,\theta)\}$,
looks like the movements of  satellites around  planets
with $0\le \theta \le 2\pi$.
For a fixed $y$, there are $p$ choices of $t$ so that $y=t^p$. 
We take $t$ so that $0\le \arg t \le 2\pi/p$.
Thus 
planets are the points $P_i=(t\eta^i,t^p)$ for
$i=0,\dots,p-1$ and the satellites around $P_i$ are
 $\{(\varphi_j(t\eta^i),t^p)\,|\,j=1,\dots,q\}$
where 
  $\eta=\exp(2\pi \sqrt{-1}/p)$.

Above conditions $(1)$ and $(2)$ say that
$p$ planets  moves an arc of the angle $2\pi /p$
centered at the origin
when $t=\varepsilon^{1/p} \exp(\sqrt{-1}\, \theta/p)$
moves from $\theta=0$ to $2\pi$.
Then the satellites,
which are the center of loops $\{a_{i,j}\,|\,j=1,\dots,q\}$,
are rotated $(p-1)$-times around $P_i$ simultaneously for
$i=0,\dots,p-1$.
Hence we have the monodromy relations:
\begin{equation*}
(1{\text{-}}1)\quad a_{i,j}=a_{i,j}^{\sigma_0}=
\begin{cases}
   \omega_{i+1}^{p-1}a_{i+1,j}\omega_{i+1}^{-(p-1)} & 0\le i\le p-2,\\
   \Omega\, \omega_{1}^{p-1}a_{0,j}(\Omega\, \omega_{1}^{p-1})^{-1} &
  i= p-1,
\end{cases}
\quad j=1\dots,q
\end{equation*}
 where
$a_{i,j}^{\sigma_0}$ is the monodromy action by $\sigma_0$ on $a_{i,j}$.
See Figure \ref{h14} for the case $p=3$ and $q=2$.\\
\begin{figure}[H]
\centering
 \input{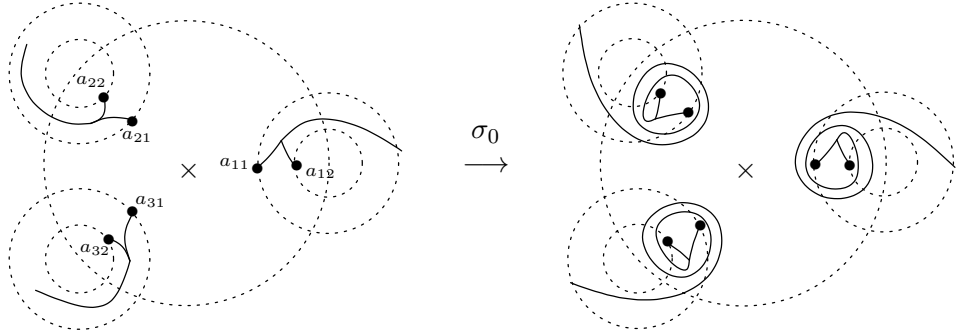}
  \caption{The case $p=3$ and $q=2$}\label{h14}
\end{figure}

On the other hand,
we get the relation $\omega_1=\omega_2=\cdots=\omega_p$
when $y=\varepsilon \exp(2\pi \sqrt{-1} \theta)$ moves around the origin once.
Hence  we have
\[
 \Omega=\omega^p,\quad
 \omega:=\omega_1=\omega_2=\cdots=\omega_p.
\]
We can rewrite the relations  $(1{\text{-}}1)$ as follows:
 \begin{equation*}
(1{\text{-}}2)\quad a_{i,j}=
\begin{cases}
   \omega^{p-1}a_{i+1,j}\omega^{-(p-1)} & 0\le i\le p-2,\\
   \omega^{2p-1}a_{0,j}\omega^{-(2p-1)} &  i= p-1,
\end{cases}
\quad j=1,\dots,q.
\end{equation*}

\noindent{\bf{Case 2}}:
Next we consider the monodromy relations at $y=\gamma_j$ for $j\ge 1$.
In this case,
the pencil line $L_{\gamma_j}$ is
tangent  to $C_j$ and 
 $C_j\cap L_{\gamma_j}=\{Q_{j,0}\}=\{(0,\gamma_j)\}$
is a flex point of $C_j$ of flex order $p-2$.
On the other hand,
the pencil line $L_{\gamma_j}$
is generic with respect to other $C_i$ for $i\ne j$.

First we consider the case $j=1$.
Recall that the defining polynomial of $C_i$ is
\[
 f_i(x,y)=x^p-\psi(y,\alpha_i)=x^p-y\prod_{k=0}^{p-2}(\gamma_i\xi^k-y),\quad
i=1,\dots,q.
\]
We take the local coordinates $(x,y_1):=(x,y-\gamma_1)$
centered at $Q_{1,0}$.
By an easy calculation,
\begin{eqnarray}\label{coefficient}
 f_{i}(x,y_1+\gamma_1)=0 \Leftrightarrow
 x^p=
 \begin{cases}
 (1-p)y_1+H_1(y_1),\quad   & i=1,\\
 \frac{\gamma_1}{\alpha_1}(\alpha_1-\alpha_i)+H_i(y_1)  &i\ne1
 \end{cases}
\end{eqnarray}
and $\alpha_1-\alpha_i>0$ where
$\ord_{y_1}\, H_1\ge 2$ and $\ord_{y_1}\,H_i\ge 1$ for $i\ge 2$. The
first coefficients $1-p$ and $\frac{\gamma_1}{\alpha_1}(\alpha_1-\alpha_i)$
are obtained from the equalities:
\[
\begin{cases}
 \psi(y_1+\gamma_1,\alpha_1)=-(y_1+\gamma_1)y_1\prod_{k\ge
   1}(\gamma_1\xi^k-y_1-\gamma_1)\\
 \psi(y_1+\gamma_1,\alpha_i)=(y_1+\gamma_1)-\alpha_i(y_1+\gamma_1)^p
\end{cases}
\]
and 
\[
\begin{cases}
\frac{d \psi(y_1+\gamma_1,\alpha_1)}{d y_1}|_{y_1=0}=\frac{d
   \psi(y,\alpha_1)}{d y}|_{y=\gamma_1}=1-p\\
\psi(\gamma_1,\alpha_i)=
\frac{\gamma_1}{\alpha_1}(\alpha_1-\alpha_i).
\end{cases}
\]
Now we consider the monodromy relations at $y=\gamma_1$.
First, the action of $\sigma_1$ on $b_1,\dots, b_q$ is
sketched as in Figure \ref{h1}.  
Thus we see that 
 the generators which are topologically deformed are
$\{a_{i,1}\,|\,i=0,\dots,p-1\}$ 
 under the rotation  $y_1=-\varepsilon\exp( \sqrt{-1}\theta)$ with
 $0\le \theta \le 2\pi$. The other generators are unchanged.
 Namely $a_{i,j}^{\sigma_1}=a_{i,j}$
 for $i=0,\dots,p-1$ and  $j\ge 2$.
To simplify the monodromy relations,
we introduce an element  $g_1:=b_2\cdots b_q$.
Then $\tau=b_1 g_1$.
See Figure \ref{h1}.
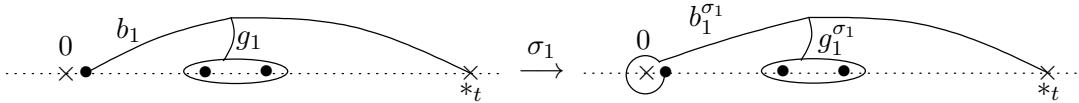
\begin{figure}[H]
\centering
 \hspace{8cm}
\unitlength 0.1in
\begin{picture}( 56.4700,  4.7500)(  4.1200,-14.1000)
%
\special{pn 8}%
\special{pa 460 1306}%
\special{pa 3036 1306}%
\special{dt 0.045}%
\put(7.7200,-13.0600){\makebox(0,0){$\times$}}%
\put(8.7600,-13.0200){\makebox(0,0){$\bullet$}}%
%
\special{pn 8}%
\special{pa 2884 1302}%
\special{pa 2856 1288}%
\special{pa 2826 1274}%
\special{pa 2796 1260}%
\special{pa 2768 1246}%
\special{pa 2738 1232}%
\special{pa 2708 1220}%
\special{pa 2680 1206}%
\special{pa 2650 1192}%
\special{pa 2620 1180}%
\special{pa 2590 1168}%
\special{pa 2560 1154}%
\special{pa 2532 1144}%
\special{pa 2502 1132}%
\special{pa 2472 1120}%
\special{pa 2442 1110}%
\special{pa 2412 1100}%
\special{pa 2380 1090}%
\special{pa 2350 1082}%
\special{pa 2320 1072}%
\special{pa 2290 1064}%
\special{pa 2258 1058}%
\special{pa 2228 1050}%
\special{pa 2196 1046}%
\special{pa 2166 1040}%
\special{pa 2134 1036}%
\special{pa 2102 1032}%
\special{pa 2070 1028}%
\special{pa 2038 1024}%
\special{pa 2006 1022}%
\special{pa 1976 1020}%
\special{pa 1944 1018}%
\special{pa 1910 1018}%
\special{pa 1878 1016}%
\special{pa 1846 1016}%
\special{pa 1814 1014}%
\special{pa 1782 1014}%
\special{pa 1750 1014}%
\special{pa 1718 1014}%
\special{pa 1684 1014}%
\special{pa 1652 1014}%
\special{pa 1636 1014}%
\special{sp}%
%
\special{pn 8}%
\special{pa 1636 1014}%
\special{pa 1606 1022}%
\special{pa 1574 1028}%
\special{pa 1542 1036}%
\special{pa 1512 1042}%
\special{pa 1480 1050}%
\special{pa 1448 1058}%
\special{pa 1418 1066}%
\special{pa 1388 1076}%
\special{pa 1356 1084}%
\special{pa 1326 1094}%
\special{pa 1296 1104}%
\special{pa 1266 1114}%
\special{pa 1236 1126}%
\special{pa 1206 1138}%
\special{pa 1178 1150}%
\special{pa 1148 1164}%
\special{pa 1120 1178}%
\special{pa 1090 1192}%
\special{pa 1062 1206}%
\special{pa 1034 1220}%
\special{pa 1004 1234}%
\special{pa 976 1250}%
\special{pa 948 1264}%
\special{pa 920 1280}%
\special{pa 890 1294}%
\special{pa 876 1302}%
\special{sp}%
\put(15.0000,-12.9800){\makebox(0,0){$\bullet$}}%
%
\special{pn 8}%
\special{ar 1660 1294 272 64  0.0000000 6.2831853}%
%
\special{pn 8}%
\special{pa 1636 1014}%
\special{pa 1646 1046}%
\special{pa 1652 1078}%
\special{pa 1652 1108}%
\special{pa 1646 1140}%
\special{pa 1636 1168}%
\special{pa 1620 1198}%
\special{pa 1604 1228}%
\special{pa 1596 1238}%
\special{sp}%
\put(17.3200,-11.3400){\makebox(0,0){$g_1$}}%
\put(11.0000,-10.7800){\makebox(0,0){$b_1$}}%
\put(28.9200,-14.0200){\makebox(0,0){$*_t$}}%
\put(7.7000,-11.6000){\makebox(0,0){$0$}}%
\put(18.2000,-12.9800){\makebox(0,0){$\bullet$}}%
%
\special{pn 8}%
\special{pa 3484 1304}%
\special{pa 6060 1304}%
\special{dt 0.045}%
\put(38.1100,-13.0300){\makebox(0,0){$\times$}}%
%
\special{pn 8}%
\special{pa 5908 1300}%
\special{pa 5878 1286}%
\special{pa 5848 1272}%
\special{pa 5820 1258}%
\special{pa 5790 1244}%
\special{pa 5760 1230}%
\special{pa 5732 1216}%
\special{pa 5702 1202}%
\special{pa 5672 1190}%
\special{pa 5644 1176}%
\special{pa 5614 1164}%
\special{pa 5584 1152}%
\special{pa 5554 1140}%
\special{pa 5524 1128}%
\special{pa 5494 1118}%
\special{pa 5464 1106}%
\special{pa 5434 1096}%
\special{pa 5404 1088}%
\special{pa 5374 1078}%
\special{pa 5342 1070}%
\special{pa 5312 1062}%
\special{pa 5282 1054}%
\special{pa 5250 1048}%
\special{pa 5220 1042}%
\special{pa 5188 1036}%
\special{pa 5156 1032}%
\special{pa 5126 1028}%
\special{pa 5094 1024}%
\special{pa 5062 1022}%
\special{pa 5030 1020}%
\special{pa 4998 1018}%
\special{pa 4966 1016}%
\special{pa 4934 1014}%
\special{pa 4902 1014}%
\special{pa 4870 1012}%
\special{pa 4838 1012}%
\special{pa 4804 1012}%
\special{pa 4772 1012}%
\special{pa 4740 1012}%
\special{pa 4708 1012}%
\special{pa 4676 1012}%
\special{pa 4660 1012}%
\special{sp}%
%
\special{pn 8}%
\special{pa 4660 1012}%
\special{pa 4628 1020}%
\special{pa 4598 1028}%
\special{pa 4566 1034}%
\special{pa 4536 1042}%
\special{pa 4504 1050}%
\special{pa 4474 1058}%
\special{pa 4442 1066}%
\special{pa 4412 1076}%
\special{pa 4380 1084}%
\special{pa 4350 1092}%
\special{pa 4320 1100}%
\special{pa 4288 1110}%
\special{pa 4258 1118}%
\special{pa 4228 1128}%
\special{pa 4196 1138}%
\special{pa 4166 1148}%
\special{pa 4136 1156}%
\special{pa 4106 1166}%
\special{pa 4074 1176}%
\special{pa 4044 1186}%
\special{pa 4014 1196}%
\special{pa 3984 1208}%
\special{pa 3954 1218}%
\special{pa 3922 1228}%
\special{pa 3892 1238}%
\special{pa 3876 1244}%
\special{sp}%
\put(45.2300,-12.9500){\makebox(0,0){$\bullet$}}%
%
\special{pn 8}%
\special{ar 4684 1292 272 64  0.0000000 6.2831853}%
%
\special{pn 8}%
\special{pa 4660 1012}%
\special{pa 4668 1044}%
\special{pa 4674 1074}%
\special{pa 4676 1106}%
\special{pa 4670 1136}%
\special{pa 4658 1166}%
\special{pa 4644 1196}%
\special{pa 4626 1224}%
\special{pa 4620 1236}%
\special{sp}%
\put(48.1000,-11.3000){\makebox(0,0){$g_1^{\sigma_1}$}}%
\put(41.3000,-10.2000){\makebox(0,0){$b_1^{\sigma_1}$}}%
\put(59.1500,-13.9900){\makebox(0,0){$*_t$}}%
\put(37.9300,-11.2700){\makebox(0,0){$0$}}%
\put(48.4300,-12.9500){\makebox(0,0){$\bullet$}}%
%
\special{pn 8}%
\special{ar 3804 1312 100 100  0.0096151 5.5589353}%
\put(32.6000,-13.0000){\makebox(0,0){$\longrightarrow$}}%
\put(32.6000,-11.8000){\makebox(0,0){$\sigma_1$}}%
\put(39.1000,-13.0000){\makebox(0,0){$\bullet$}}%
\put(59.1000,-13.0000){\makebox(0,0){$\times$}}%
\put(28.9000,-13.0000){\makebox(0,0){$\times$}}%
\end{picture}%
 \vspace{0.5cm}
   \caption{$\{y=\gamma_1-\varepsilon\}\cap \Bbb{C}^2_t$}\label{h1}
\end{figure}
\noindent
Let $g_{i,1}$ be the pull-back of $g_1$
starting from $*_i$ for $i=0,\dots,p-1$.
More precisely,  
$g_{i,1}=a_{i,2}\cdots a_{i,q}$ and
$\omega_i=a_{i,1}g_{i,1}$.\\
When $y_1=y-\gamma_1=-\varepsilon\exp( \sqrt{-1}\theta)$
moves from  $\theta=0$ to $2\pi$,
the generators $a_{0,1},\dots,a_{p-1,1}$ moves  an arc of the angle
$2\pi /p$ centered at the origin (the lifts of $b_1^{\sigma_1}$)
and the other generators do not move.
Thus we have following monodromy relations:
\[
(2{\text{-}}1)\quad
a_{i,1}=a_{i,1}^{\sigma_1}=
       \begin{cases}
        g_{i+1,1}^{-1}\, a_{i+1,1}\, g_{i+1,1}  & 0\le i\le p-2\\
        \Omega \,g_{0,1}^{-1}\,a_{0,1}\,(\Omega\, g_{0,1}^{-1})^{-1} & i = p-1
        \end{cases}
 \]
 and $g_{i,1}^{\sigma_1}=g_{i,1}$ for $i=0,\dots, p-1$.
 See  Figure \ref{h2}.
\begin{figure}[H]
\centering
\input{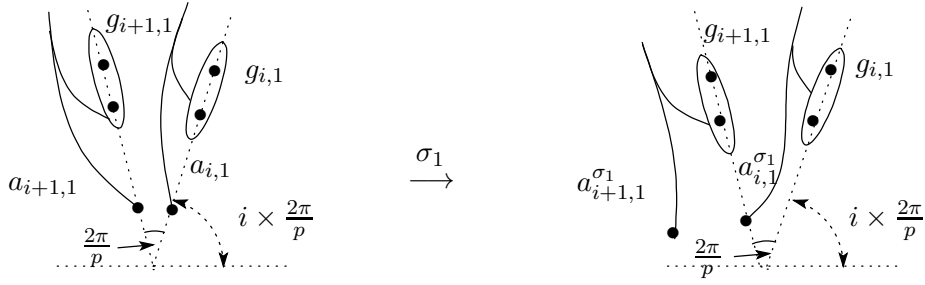}
 \vspace{0.3cm}
  \caption{The action of $\sigma_1$}\label{h2}
\end{figure}
\noindent
By the previous argument, we have $\omega_1=\omega$ and $\Omega=\omega^p$.
Hence we can rewrite the relations  $(2{\text{-}}1)$ as follows:
\[
(2{\text{-}}2)\quad
a_{i,1}=a_{i,1}^{\sigma_1}=
       \begin{cases}
        \omega^{-1} a_{i+1,1} \omega  & 0\le i\le p-2\\
        \omega^{p-1} a_{0,1} \omega^{-(p-1)}& i=p-1\\
        \end{cases}
 \]


Now we consider the case $j\ge 2$.
First we deform the pencil from $\gamma_0$ to $\gamma_j-\varepsilon$
along $K_j$.
Note that  
\[
\psi(y,\alpha_k)
 \begin{cases}
  <0 & k < j\\
  >0 & k > j.
 \end{cases}
\]
where  $y\in [\gamma_{j-1}+\varepsilon,  \gamma_{j}-\varepsilon]$.
Thus the generators $b_1,\dots,b_q$ are deformed as in Figure \ref{h9}  where $\psi_j:=\psi(y,\alpha_j)$.
\begin{figure}[H]
\hspace{-1cm}
 \input{./q8.tex}
 \vspace{0.2cm}
 \caption{$\{y=\gamma_j-\varepsilon\}\cap \Bbb{C}^2_t$}\label{h9}
\end{figure}
\noindent
When $y$ moves along $S_j: |y-\gamma_j|=\varepsilon$,
the single root of $g_{\alpha_j}(x,y)=0$
which is near the origin
goes around the origin once and the other roots $g_{\alpha_k}(x,y)=0\
(k\ne j)$
do not move as in Figure \ref{h10}  where $\psi_j:=\psi(y,\alpha_j)$.
\begin{figure}[H]
\hspace{-1cm}
 \input{./q9.tex}
 \vspace{0.3cm}
 \caption{$\{y=\gamma_j-\varepsilon\}\cap \Bbb{C}^2_t$}\label{h10}
\end{figure}
\noindent
This implies, by taking  $p$-fold covering,
   the corresponding generators $a_{0,j},\dots,a_{p-1,j}$ of $b_j$
   moves  an arc of the angle $2\pi /p$ centered at the origin.\\
To see it more precisely, we put new loops:
 \begin{alignat*}{1}
  h_{j}=\begin{cases}
            e                          & j=1   \\
            b_{1}\cdots b_{j-1}    & 2\le j \le q
         \end{cases},
 &\qquad
  g_{j}=\begin{cases}
            b_{j+1}\cdots b_{q}    & 1\le j \le q-1\\
            e                          & j=q.
          \end{cases}
\end{alignat*}
By the definitions,
we have $\tau=h_{j}b_{j}g_{j}$.
See Figure \ref{g-10}.
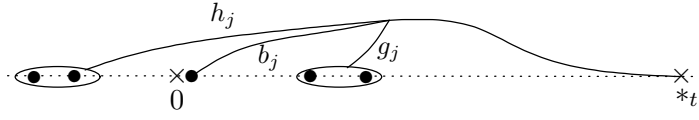
\begin{figure}[H]
\centering
\unitlength 0.1in
\begin{picture}( 43.0100,  4.5500)(  8.0900,-19.1600)
%
\special{pn 8}%
\special{pa 1524 1858}%
\special{pa 5110 1862}%
\special{dt 0.045}%
%
\special{pn 8}%
\special{pa 3522 1568}%
\special{pa 3504 1596}%
\special{pa 3488 1622}%
\special{pa 3470 1650}%
\special{pa 3452 1676}%
\special{pa 3432 1700}%
\special{pa 3412 1724}%
\special{pa 3388 1746}%
\special{pa 3364 1768}%
\special{pa 3340 1788}%
\special{pa 3314 1808}%
\special{pa 3302 1816}%
\special{sp}%
%
\special{pn 8}%
\special{pa 3522 1568}%
\special{pa 3490 1570}%
\special{pa 3458 1574}%
\special{pa 3426 1576}%
\special{pa 3394 1580}%
\special{pa 3362 1582}%
\special{pa 3330 1586}%
\special{pa 3298 1588}%
\special{pa 3266 1592}%
\special{pa 3236 1594}%
\special{pa 3204 1598}%
\special{pa 3172 1600}%
\special{pa 3140 1604}%
\special{pa 3108 1606}%
\special{pa 3076 1610}%
\special{pa 3044 1614}%
\special{pa 3012 1616}%
\special{pa 2980 1620}%
\special{pa 2948 1624}%
\special{pa 2916 1626}%
\special{pa 2884 1630}%
\special{pa 2852 1634}%
\special{pa 2820 1638}%
\special{pa 2788 1642}%
\special{pa 2756 1644}%
\special{pa 2724 1648}%
\special{pa 2694 1652}%
\special{pa 2662 1656}%
\special{pa 2630 1662}%
\special{pa 2598 1666}%
\special{pa 2566 1670}%
\special{pa 2534 1676}%
\special{pa 2502 1680}%
\special{pa 2470 1686}%
\special{pa 2440 1692}%
\special{pa 2408 1696}%
\special{pa 2376 1704}%
\special{pa 2344 1710}%
\special{pa 2314 1716}%
\special{pa 2282 1722}%
\special{pa 2252 1730}%
\special{pa 2220 1738}%
\special{pa 2190 1746}%
\special{pa 2158 1754}%
\special{pa 2128 1762}%
\special{pa 2096 1772}%
\special{pa 2066 1782}%
\special{pa 2036 1790}%
\special{pa 2006 1802}%
\special{pa 1976 1812}%
\special{pa 1946 1822}%
\special{pa 1930 1828}%
\special{sp}%
%
\special{pn 8}%
\special{pa 3522 1568}%
\special{pa 3490 1574}%
\special{pa 3458 1580}%
\special{pa 3428 1586}%
\special{pa 3396 1592}%
\special{pa 3364 1598}%
\special{pa 3334 1604}%
\special{pa 3302 1610}%
\special{pa 3270 1616}%
\special{pa 3238 1620}%
\special{pa 3208 1626}%
\special{pa 3176 1630}%
\special{pa 3144 1636}%
\special{pa 3112 1640}%
\special{pa 3080 1646}%
\special{pa 3048 1650}%
\special{pa 3016 1654}%
\special{pa 2984 1660}%
\special{pa 2952 1666}%
\special{pa 2920 1672}%
\special{pa 2888 1678}%
\special{pa 2858 1686}%
\special{pa 2826 1694}%
\special{pa 2796 1702}%
\special{pa 2766 1712}%
\special{pa 2736 1722}%
\special{pa 2706 1734}%
\special{pa 2678 1748}%
\special{pa 2648 1762}%
\special{pa 2620 1776}%
\special{pa 2594 1792}%
\special{pa 2566 1810}%
\special{pa 2538 1826}%
\special{pa 2512 1844}%
\special{pa 2490 1858}%
\special{sp}%
\put(18.7400,-18.6600){\makebox(0,0){{\large{$\bullet$}}}}%
\put(24.8900,-18.6600){\makebox(0,0){{\large{$\bullet$}}}}%
\put(31.0900,-18.6700){\makebox(0,0){{\large{$\bullet$}}}}%
%
\special{pn 8}%
\special{pa 3522 1568}%
\special{pa 3554 1566}%
\special{pa 3588 1564}%
\special{pa 3620 1564}%
\special{pa 3652 1564}%
\special{pa 3686 1564}%
\special{pa 3718 1564}%
\special{pa 3750 1564}%
\special{pa 3782 1566}%
\special{pa 3814 1570}%
\special{pa 3846 1574}%
\special{pa 3876 1580}%
\special{pa 3908 1586}%
\special{pa 3938 1594}%
\special{pa 3968 1602}%
\special{pa 3998 1614}%
\special{pa 4026 1626}%
\special{pa 4056 1640}%
\special{pa 4084 1654}%
\special{pa 4114 1668}%
\special{pa 4142 1684}%
\special{pa 4170 1698}%
\special{pa 4200 1712}%
\special{pa 4228 1726}%
\special{pa 4258 1738}%
\special{pa 4288 1750}%
\special{pa 4318 1762}%
\special{pa 4348 1772}%
\special{pa 4378 1782}%
\special{pa 4408 1790}%
\special{pa 4438 1798}%
\special{pa 4470 1806}%
\special{pa 4500 1812}%
\special{pa 4532 1818}%
\special{pa 4564 1824}%
\special{pa 4596 1828}%
\special{pa 4628 1834}%
\special{pa 4660 1838}%
\special{pa 4692 1840}%
\special{pa 4724 1844}%
\special{pa 4756 1846}%
\special{pa 4788 1850}%
\special{pa 4820 1852}%
\special{pa 4852 1854}%
\special{pa 4886 1856}%
\special{pa 4918 1858}%
\special{pa 4950 1858}%
\special{pa 4984 1860}%
\special{pa 5016 1862}%
\special{pa 5046 1862}%
\special{sp}%
\put(50.8000,-19.7000){\makebox(0,0){$*_t$}}%
\put(26.5400,-15.4600){\makebox(0,0){{\small{$h_j$}}}}%
\put(28.9200,-17.6600){\makebox(0,0){{\small{$b_j$}}}}%
\put(35.2100,-17.3500){\makebox(0,0){{\small{$g_j$}}}}%
\put(24.1100,-19.8400){\makebox(0,0){$0$}}%
\put(24.1100,-18.5700){\makebox(0,0){$\times$}}%
\put(34.0000,-18.6900){\makebox(0,0){{\large{$\bullet$}}}}%
%
\special{pn 8}%
\special{ar 3260 1864 222 54  0.0000000 6.2831853}%
%
\special{pn 8}%
\special{ar 1786 1864 222 54  0.0000000 6.2831853}%
\put(16.6400,-18.6900){\makebox(0,0){{\large{$\bullet$}}}}%
\put(50.5000,-18.6000){\makebox(0,0){$\times$}}%
\end{picture}%
\vspace{0.5cm}
 \caption{New loops}\label{g-10}
\end{figure}
\noindent
We take the local coordinates
$(x,y_j):=(x,y-\gamma_j)$ centered at $Q_{j,0}$.
Then
\[
 f_{i}(x,y_j+\gamma_j)=0
 \Leftrightarrow
 x^p=
 \begin{cases}
 (1-p)y_j+H_j(y_j),\quad   & i=j,\\
 \frac{\gamma_j}{\alpha_j}(\alpha_j-\alpha_i)+H_i(y_j)  &i\ne j
 \end{cases}
\]
where
$\ord_{y_j}\, H_j\ge 2$ and $\ord_{y_j}\,H_i\ge 1$ for $i\ne j$.
By the assumption, we have
$\alpha_j-\alpha_i>0$ or
$\alpha_j-\alpha_i<0$ corresponding to either $i >j$ or $i<j$
respectively.
Thus we can see that 
the generators which are deformed under this monodromy are
$\{a_{i,j}\,|\,i=0,\dots,p-1\}$
when $y_j$ moves around the circle $|y_j|=\varepsilon$.
Thus  $a_{i,k}^{\sigma_j}=a_{i,k}$ for $k\ne j$.\\
Let $h_{i,j}$ and $g_{i,j}$ be the pull-back of
    $h_j$ and $g_j$ respectively.
By the definition, we have
    \begin{alignat*}{1}
  h_{i,j}= a_{i,1}\cdots a_{i,j-1}\\
 g_{i,j}= a_{i,j+1}\cdots a_{i,q} 
\end{alignat*}
where $h_{i,1}=e$ and $g_{i,q}=e$
and we put  $\omega_i=h_{i,j}a_{i,j}g_{i,j}$.

When $y$ 
moves around the circle $|y-\gamma_j|=\varepsilon$ once, 
the generators $a_{0,j},\dots,a_{p-1,j}$ moves  an arc of the angle
$2\pi /p$ centered at the origin.
Thus  we have following monodromy relations:
\[
(2{\text{-}}3)\quad
a_{i,j}=a_{i,j}^{\sigma_j}=
       \begin{cases}
       ( g_{i+1,j}h_{i,j})^{-1}\, a_{i+1,j}\, g_{i+1,j}h_{i,j}  & 0\le i\le p-2\\
     h_{p-1,j}^{-1}\,\Omega\, g_{0,j}^{-1}\,a_{0,j}\,
        ( h_{p-1,j}^{-1}\Omega g_{0,j}^{-1})^{-1}
    & i = p-1
        \end{cases}
 \]
 and $g_{i,j}^{\sigma_j}=g_{i,j}$ and  $h_{i,j}^{\sigma_j}=h_{i,j}$
 for $i=0,\dots,p-1$.
 See  Figure \ref{h4}.
\begin{figure}[H]
 \input{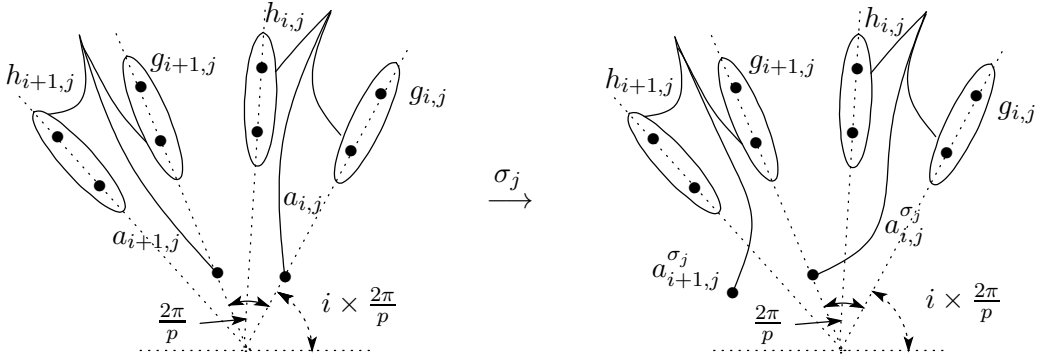}
 \vspace{0.3cm}
  \caption{The action of $\sigma_j$}\label{h4}
\end{figure}

\noindent{\bf{Other cases}}:
Finally we read the monodromy relations at
$y=\gamma$ where $\gamma\in \Sigma$ with
$\gamma\ne 0$, $\gamma_1,\dots,\gamma_q$.
Recall that 
the set of parameters
$\Sigma \subset \Bbb{C}_y$ which corresponds to singular pencils are
given by
\[
 \Sigma=\{0,\gamma_{j}\xi^{k}\in \Bbb{C}_y\,|\,
 k=0,\dots,p-2,\ j=1,\dots,q  \},\ \
\xi=\exp\left(\frac{2\pi \sqrt{-1}}{p-1}\right).
\]
Then the pencil line $L_{\gamma_{j}\xi^{k}}=\{y=\gamma_{j}\xi^{k}\}$ is
singular with respect to $C_j$ and 
$C_j\cap L_{\gamma_{j}\xi^{k}}=\{Q_{j,k}\}
           =\{(0,\gamma_j\xi^{k})\}$
is a flex point of $C_j$ of flex order $p-2$ for $k=1,\dots, p-2$.
Note that the pencil line $L_{\gamma_{j}\xi^{k}}$ is generic
with respect to other $C_i$ for $i\ne j$.

First we consider the case $k=1$.
That is, we consider the monodromy relations at $y=\gamma_j \xi$.
We take a path $L_1$ which connects $\gamma_0$ and $\gamma_0\xi$
as in Figure \ref{d6}.
\begin{figure}[H]
\centering
\unitlength 0.1in
\begin{picture}( 27.0000,  7.7000)(  6.7000,-19.7500)
%
\special{pn 8}%
\special{pa 670 1900}%
\special{pa 3370 1900}%
\special{dt 0.045}%
%
\special{pn 13}%
\special{ar 1916 1900 420 420  5.7158220 6.2831853}%
%
\special{pn 13}%
\special{pa 2336 1900}%
\special{pa 2636 1900}%
\special{fp}%
%
\special{pn 13}%
\special{pa 2270 1680}%
\special{pa 2526 1524}%
\special{fp}%
\put(26.8000,-20.6000){\makebox(0,0){$\gamma_0$}}%
\put(26.3500,-18.9900){\makebox(0,0){$\bullet$}}%
\put(25.5000,-15.1000){\makebox(0,0){$\bullet$}}%
\put(25.7000,-13.4000){\makebox(0,0){$\gamma_0 \xi$}}%
\put(25.6000,-17.5000){\makebox(0,0){$L_1$}}%
\put(19.2000,-20.6000){\makebox(0,0){$0$}}%
\put(19.3000,-18.9900){\makebox(0,0){$\times$}}%
%
\special{pn 8}%
\special{pa 1930 1900}%
\special{pa 2996 1206}%
\special{dt 0.045}%
%
\special{pn 8}%
\special{pa 2300 1730}%
\special{pa 2354 1772}%
\special{fp}%
%
\special{pn 8}%
\special{pa 2300 1730}%
\special{pa 2274 1792}%
\special{fp}%
\end{picture}%
\vspace{0.5cm}
 \caption{The loop $L_1$}\label{d6}
\end{figure}
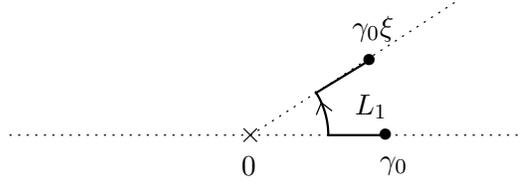
\noindent
Then the loops $b_1,\dots,b_q$ are deformed
as in the left side of  Figure \ref{d1}.
We take new loops $c_1,\dots, c_q$ as in the right side of  Figure
\ref{d1}.
Here $\xi *_t=(\xi \omega_0^p,\xi \gamma_0)$.
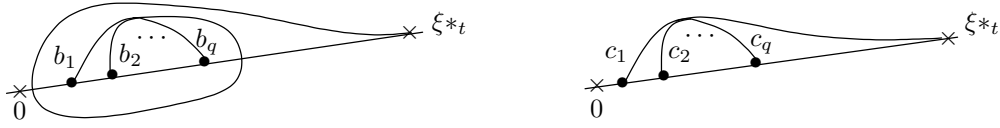
\begin{figure}[H]
\centering
{\small{
\unitlength 0.1in
\begin{picture}( 52.6700,  5.9900)(  8.8000,-20.7600)
%
\special{pn 8}%
\special{pa 3288 1632}%
\special{pa 3256 1636}%
\special{pa 3224 1638}%
\special{pa 3192 1642}%
\special{pa 3160 1642}%
\special{pa 3128 1644}%
\special{pa 3096 1644}%
\special{pa 3064 1642}%
\special{pa 3032 1640}%
\special{pa 3000 1638}%
\special{pa 2968 1634}%
\special{pa 2936 1632}%
\special{pa 2904 1626}%
\special{pa 2874 1622}%
\special{pa 2842 1616}%
\special{pa 2810 1612}%
\special{pa 2778 1606}%
\special{pa 2746 1598}%
\special{pa 2714 1592}%
\special{pa 2682 1586}%
\special{pa 2652 1580}%
\special{pa 2620 1572}%
\special{pa 2588 1566}%
\special{pa 2556 1560}%
\special{pa 2526 1554}%
\special{pa 2494 1546}%
\special{pa 2462 1542}%
\special{pa 2432 1536}%
\special{pa 2400 1530}%
\special{pa 2370 1526}%
\special{pa 2338 1522}%
\special{pa 2308 1516}%
\special{pa 2276 1512}%
\special{pa 2244 1510}%
\special{pa 2214 1506}%
\special{pa 2182 1502}%
\special{pa 2150 1500}%
\special{pa 2118 1496}%
\special{pa 2086 1494}%
\special{pa 2052 1492}%
\special{pa 2020 1488}%
\special{pa 1986 1486}%
\special{pa 1952 1484}%
\special{pa 1918 1482}%
\special{pa 1884 1480}%
\special{pa 1850 1478}%
\special{pa 1816 1478}%
\special{pa 1780 1478}%
\special{pa 1746 1478}%
\special{pa 1714 1480}%
\special{pa 1680 1484}%
\special{pa 1648 1490}%
\special{pa 1616 1496}%
\special{pa 1586 1506}%
\special{pa 1556 1516}%
\special{pa 1526 1530}%
\special{pa 1498 1544}%
\special{pa 1472 1562}%
\special{pa 1446 1582}%
\special{pa 1424 1606}%
\special{pa 1402 1632}%
\special{pa 1380 1660}%
\special{pa 1362 1690}%
\special{pa 1346 1722}%
\special{pa 1332 1756}%
\special{pa 1322 1788}%
\special{pa 1312 1820}%
\special{pa 1306 1854}%
\special{pa 1304 1884}%
\special{pa 1304 1914}%
\special{pa 1308 1940}%
\special{pa 1314 1964}%
\special{pa 1326 1986}%
\special{pa 1340 2004}%
\special{pa 1356 2022}%
\special{pa 1376 2036}%
\special{pa 1398 2046}%
\special{pa 1422 2056}%
\special{pa 1450 2064}%
\special{pa 1480 2070}%
\special{pa 1510 2074}%
\special{pa 1544 2076}%
\special{pa 1580 2078}%
\special{pa 1616 2076}%
\special{pa 1654 2074}%
\special{pa 1692 2072}%
\special{pa 1732 2066}%
\special{pa 1772 2062}%
\special{pa 1814 2056}%
\special{pa 1856 2048}%
\special{pa 1898 2040}%
\special{pa 1940 2032}%
\special{pa 1982 2024}%
\special{pa 2024 2016}%
\special{pa 2064 2006}%
\special{pa 2106 1998}%
\special{pa 2146 1988}%
\special{pa 2184 1980}%
\special{pa 2220 1968}%
\special{pa 2256 1958}%
\special{pa 2288 1944}%
\special{pa 2316 1930}%
\special{pa 2342 1914}%
\special{pa 2364 1894}%
\special{pa 2380 1872}%
\special{pa 2392 1846}%
\special{pa 2400 1818}%
\special{pa 2402 1786}%
\special{pa 2400 1752}%
\special{pa 2392 1720}%
\special{pa 2382 1692}%
\special{pa 2368 1666}%
\special{pa 2350 1642}%
\special{pa 2330 1622}%
\special{pa 2308 1606}%
\special{pa 2282 1592}%
\special{pa 2254 1580}%
\special{pa 2224 1570}%
\special{pa 2192 1562}%
\special{pa 2160 1556}%
\special{pa 2124 1552}%
\special{pa 2088 1550}%
\special{pa 2050 1548}%
\special{pa 2012 1548}%
\special{pa 1972 1548}%
\special{pa 1932 1550}%
\special{pa 1890 1552}%
\special{pa 1850 1554}%
\special{pa 1832 1554}%
\special{sp}%
%
\special{pn 8}%
\special{pa 1170 1950}%
\special{pa 3348 1630}%
\special{fp}%
\put(34.8900,-15.8300){\makebox(0,0){$\xi *_t$}}%
\put(15.1100,-18.9400){\makebox(0,0){$\bullet$}}%
\put(17.2400,-18.5500){\makebox(0,0){$\bullet$}}%
\put(22.0700,-17.8700){\makebox(0,0){$\bullet$}}%
%
\special{pn 8}%
\special{pa 1864 1554}%
\special{pa 1832 1556}%
\special{pa 1802 1560}%
\special{pa 1774 1568}%
\special{pa 1748 1582}%
\special{pa 1722 1598}%
\special{pa 1698 1616}%
\special{pa 1676 1640}%
\special{pa 1654 1664}%
\special{pa 1634 1690}%
\special{pa 1614 1720}%
\special{pa 1596 1750}%
\special{pa 1578 1782}%
\special{pa 1560 1816}%
\special{pa 1542 1850}%
\special{pa 1524 1882}%
\special{pa 1522 1886}%
\special{sp}%
%
\special{pn 8}%
\special{pa 1864 1554}%
\special{pa 1832 1560}%
\special{pa 1800 1568}%
\special{pa 1776 1582}%
\special{pa 1754 1600}%
\special{pa 1740 1624}%
\special{pa 1728 1652}%
\special{pa 1720 1682}%
\special{pa 1716 1716}%
\special{pa 1712 1754}%
\special{pa 1712 1792}%
\special{pa 1712 1832}%
\special{pa 1714 1866}%
\special{sp}%
%
\special{pn 8}%
\special{pa 1864 1554}%
\special{pa 1896 1564}%
\special{pa 1928 1572}%
\special{pa 1958 1582}%
\special{pa 1988 1594}%
\special{pa 2018 1606}%
\special{pa 2046 1620}%
\special{pa 2072 1636}%
\special{pa 2098 1654}%
\special{pa 2122 1676}%
\special{pa 2146 1698}%
\special{pa 2168 1720}%
\special{pa 2190 1744}%
\special{pa 2212 1770}%
\special{pa 2226 1788}%
\special{sp}%
%
\special{pn 8}%
\special{pa 4198 1926}%
\special{pa 6148 1656}%
\special{fp}%
\put(62.7700,-15.8500){\makebox(0,0){$\xi *_t$}}%
\put(43.9500,-18.9800){\makebox(0,0){$\bullet$}}%
\put(46.0900,-18.5900){\makebox(0,0){$\bullet$}}%
\put(50.9100,-17.9200){\makebox(0,0){$\bullet$}}%
%
\special{pn 8}%
\special{pa 4750 1560}%
\special{pa 4718 1560}%
\special{pa 4688 1564}%
\special{pa 4660 1574}%
\special{pa 4632 1586}%
\special{pa 4608 1602}%
\special{pa 4584 1622}%
\special{pa 4562 1644}%
\special{pa 4540 1668}%
\special{pa 4520 1696}%
\special{pa 4500 1724}%
\special{pa 4480 1756}%
\special{pa 4462 1788}%
\special{pa 4444 1820}%
\special{pa 4426 1854}%
\special{pa 4408 1888}%
\special{pa 4406 1890}%
\special{sp}%
%
\special{pn 8}%
\special{pa 4750 1560}%
\special{pa 4716 1566}%
\special{pa 4686 1574}%
\special{pa 4660 1586}%
\special{pa 4640 1606}%
\special{pa 4624 1628}%
\special{pa 4612 1656}%
\special{pa 4604 1688}%
\special{pa 4600 1722}%
\special{pa 4598 1758}%
\special{pa 4598 1796}%
\special{pa 4598 1836}%
\special{pa 4598 1870}%
\special{sp}%
%
\special{pn 8}%
\special{pa 4750 1560}%
\special{pa 4782 1568}%
\special{pa 4812 1578}%
\special{pa 4844 1588}%
\special{pa 4874 1598}%
\special{pa 4902 1610}%
\special{pa 4930 1624}%
\special{pa 4958 1640}%
\special{pa 4984 1660}%
\special{pa 5008 1680}%
\special{pa 5030 1702}%
\special{pa 5054 1724}%
\special{pa 5074 1748}%
\special{pa 5096 1774}%
\special{pa 5112 1792}%
\special{sp}%
%
\special{pn 8}%
\special{pa 4674 1572}%
\special{pa 4706 1562}%
\special{pa 4736 1556}%
\special{pa 4768 1552}%
\special{pa 4800 1550}%
\special{pa 4832 1552}%
\special{pa 4864 1554}%
\special{pa 4896 1558}%
\special{pa 4928 1562}%
\special{pa 4960 1568}%
\special{pa 4992 1574}%
\special{pa 5024 1580}%
\special{pa 5056 1586}%
\special{pa 5086 1592}%
\special{pa 5118 1598}%
\special{pa 5150 1606}%
\special{pa 5180 1612}%
\special{pa 5212 1618}%
\special{pa 5242 1624}%
\special{pa 5274 1630}%
\special{pa 5306 1634}%
\special{pa 5338 1640}%
\special{pa 5370 1644}%
\special{pa 5400 1648}%
\special{pa 5432 1652}%
\special{pa 5464 1654}%
\special{pa 5496 1658}%
\special{pa 5528 1660}%
\special{pa 5560 1662}%
\special{pa 5592 1664}%
\special{pa 5624 1666}%
\special{pa 5656 1668}%
\special{pa 5688 1668}%
\special{pa 5720 1670}%
\special{pa 5754 1670}%
\special{pa 5786 1670}%
\special{pa 5818 1670}%
\special{pa 5850 1668}%
\special{pa 5882 1668}%
\special{pa 5914 1666}%
\special{pa 5946 1666}%
\special{pa 5978 1664}%
\special{pa 6010 1664}%
\special{pa 6042 1662}%
\special{pa 6064 1660}%
\special{sp}%
\put(14.7800,-17.5900){\makebox(0,0){$b_1$}}%
\put(18.2000,-17.4900){\makebox(0,0){$b_2$}}%
\put(22.1900,-16.8000){\makebox(0,0){$b_q$}}%
\put(43.6300,-17.4800){\makebox(0,0){$c_1$}}%
\put(46.9400,-17.4800){\makebox(0,0){$c_2$}}%
\put(51.2200,-16.9300){\makebox(0,0){$c_q$}}%
\put(19.4000,-16.6000){\special{rt 0 0  6.1519}\makebox(0,0){$\cdots$}}%
\special{rt 0 0 0}%
\put(48.1000,-16.5000){\special{rt 0 0  6.1439}\makebox(0,0){$\cdots$}}%
\special{rt 0 0 0}%
\put(12.4000,-20.4000){\makebox(0,0){$0$}}%
\put(42.6000,-20.3000){\makebox(0,0){$0$}}%
\put(12.4000,-19.4000){\makebox(0,0){$\times$}}%
\put(42.6000,-19.1000){\makebox(0,0){$\times$}}%
\put(61.0000,-16.6000){\makebox(0,0){$\times$}}%
\put(32.8000,-16.3000){\makebox(0,0){$\times$}}%
\end{picture}%
 }}
 \caption{New loops $c_1,\dots, c_q$.}\label{d1}
\end{figure}
\noindent
They are related by the following.
\begin{eqnarray}\label{a1}
c_j=\tau^{-1}b_j \tau,\quad j=1,\dots,q. 
\end{eqnarray}
Let $d_{0,j},\dots,d_{p-1,j}$ be the pull-back of $c_j$ by $\varphi_p$ for
$j=1,\dots,q$. 
Then the relation (\ref{a1}) implies 
\begin{eqnarray}\label{a2}
d_{i,j}=\omega^{-1}a_{i,j}\omega.
\end{eqnarray} 
Now we consider the loops 
$\sigma_1^{(1)},\dots,\sigma_q^{(1)}$ in
$\Bbb{C}_y$ with base point $\gamma_0\xi$ as in Figure
\ref{d2}.
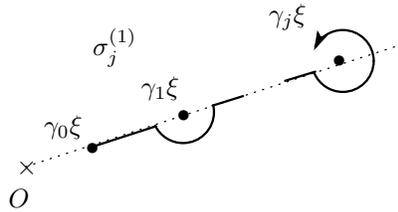
\begin{figure}[H]
\centering
 {\small{
\unitlength 0.1in
\begin{picture}( 24.0800,  9.6200)( 14.6000,-27.1500)
\put(19.0000,-26.4000){\special{rt 0 0  6.2648}\makebox(0,0){$\times$}}%
\special{rt 0 0 0}%
\put(18.6000,-28.0000){\makebox(0,0){$O$}}%
%
\special{pn 8}%
\special{pa 1894 2636}%
\special{pa 3868 1994}%
\special{dt 0.045}%
\put(35.3600,-20.8500){\special{rt 0 0  0.1498}\makebox(0,0){$\bullet$}}%
\special{rt 0 0 0}%
%
\special{pn 13}%
\special{pa 2236 2534}%
\special{pa 2576 2424}%
\special{fp}%
%
\special{pn 13}%
\special{ar 2718 2350 162 162  6.0584741 6.2831853}%
\special{ar 2718 2350 162 162  0.0000000 2.6885267}%
\put(32.6500,-18.4900){\special{rt 0 0  6.2648}\makebox(0,0){$\gamma_j\xi$}}%
\special{rt 0 0 0}%
\put(23.6000,-20.1000){\makebox(0,0){$\sigma_j^{(1)}$}}%
\put(20.9000,-24.2800){\makebox(0,0){$\gamma_0 \xi$}}%
\put(22.4400,-25.4000){\makebox(0,0){$\bullet$}}%
%
\special{pn 13}%
\special{ar 3554 2082 162 162  3.6168228 6.2831853}%
\special{ar 3554 2082 162 162  0.0000000 2.7306820}%
%
\special{pn 13}%
\special{pa 3416 1998}%
\special{pa 3412 2008}%
\special{fp}%
\special{sh 1}%
\special{pa 3412 2008}%
\special{pa 3460 1956}%
\special{pa 3436 1960}%
\special{pa 3424 1938}%
\special{pa 3412 2008}%
\special{fp}%
%
\special{pn 13}%
\special{pa 2872 2316}%
\special{pa 3032 2266}%
\special{fp}%
%
\special{pn 13}%
\special{pa 3252 2186}%
\special{pa 3410 2138}%
\special{fp}%
\put(27.2100,-23.6700){\special{rt 0 0  0.3129}\makebox(0,0){$\bullet$}}%
\special{rt 0 0 0}%
\put(25.9800,-22.2500){\makebox(0,0){$\gamma_1\xi$}}%
\end{picture}%
}}
 \vspace{0.5cm}
  \caption{The loop $\sigma_j^{(1)}$}\label{d2}
\end{figure}
\noindent
We will see that  the monodromy relations are exactly as
$(2{\text{-}}3)$.
To see this assertion, we take the following modified  coordinates $(\tilde{x},\tilde{y})$
which are defined by 
\begin{eqnarray*}
\tilde{x}:=\exp\left(\frac{-2\pi \sqrt{-1}}{p(p-1)}\right)x,\quad
\tilde{y}:=\xi y.
\end{eqnarray*}  
In these  coordinates, the loops $\sigma_1^{(1)},\dots, \sigma_q^{(1)}$ 
coincide with $\sigma_1,\dots, \sigma_q$ and 
$C_j$ is defined by the same equality:  
\[
 C_j:\quad \tilde{x}^p=\tilde{y}(1-\alpha_j\tilde{y}^{p-1}).
\]
The situation of loops  $c_{1},\dots,c_q$ are the same  with that of
$b_1,\dots,b_q$ and
the situation of loops
$d_{i,j},i=0,\dots,p-1, j=1,\dots,q$ are the same with that of
$a_{i,j},i=0,\dots,p-1, j=1,\dots,q$.
 Therefore we obtain the relations
\begin{equation*}
 \begin{split}
(2{\text{-}}3)'&\quad
d_{i,j}=
       \begin{cases}
       ( \tilde{g}_{i+1,j}\tilde{h}_{i,j})^{-1}\, d_{i+1,j}\, 
\tilde{g}_{i+1,j}\tilde{h}_{i,j}  & 0\le i\le p-2\\
     \tilde{h}_{p-1,j}^{-1}\,\tilde{\Omega}\, \tilde{g}_{0,j}^{-1}\,d_{0,j}\,
        ( \tilde{h}_{p-1,j}^{-1}\tilde{\Omega} \tilde{g}_{0,j}^{-1})^{-1}
    & i = p-1
        \end{cases},\quad j=1,\dots,q\\
\end{split}
\end{equation*}
where $\tilde{h}_{i,j}:=d_{i,1}\cdots d_{i,j-1}$,
      $\tilde{g}_{i,j}:=d_{i,j+1}\cdots d_{i,q}$ and
      $\tilde{\Omega}:=\omega^{-1}\Omega \omega$.
Now we claim the following.
\begin{lemma}\label{a3}
The relation  $(2{\text{-}}3)'$ is the same with 
the relation  $(2{\text{-}}3)$.
\end{lemma} 
\begin{proof}
First we consider the relation
 $d_{i,j}=(\tilde{g}_{i+1,j}\tilde{h}_{i,j})^{-1}\, d_{i+1,j}\,
          \tilde{g}_{i+1,j}\tilde{h}_{i,j}$ in 
  $(2{\text{-}}3)'$. 
By the relation (\ref{a2}), we have
 \begin{eqnarray*}
   \tilde{h}_{i,j}=\omega^{-1}h_{i,j}\omega,\quad
   \tilde{g}_{i,j}=\omega^{-1}g_{i,j}\omega.
  \end{eqnarray*}
Thus $d_{i,j}=(\tilde{g}_{i+1,j}\tilde{h}_{i,j})^{-1}\, d_{i+1,j}\,
          \tilde{g}_{i+1,j}\tilde{h}_{i,j}$ can be translated as follows
 \begin{eqnarray*}
  \begin{split}
 d_{i,j}=\omega^{-1}a_{i,j}\omega
   &=(\tilde{g}_{i+1,j}\tilde{h}_{i,j})^{-1}\, d_{i+1,j}\,
                     \tilde{g}_{i+1,j}\tilde{h}_{i,j}\\
   &=\left((\omega^{-1}g_{i+1,j}\omega)(\omega^{-1}h_{i,j}\omega)\right)^{-1}
           (\omega^{-1}a_{i+1,j}\omega)(\omega^{-1}g_{i+1,j}\omega)
           (\omega^{-1}h_{i,j}\omega)\\
   &=\omega^{-1}(g_{i+1,j}h_{i,j})^{-1}a_{i+1,j}g_{i+1,j}h_{i,j}\omega
  \end{split}
 \end{eqnarray*}
 which implies (2-3).
 For the relation
 $d_{i,j}=\tilde{h}_{p-1,j}^{-1}\,\tilde{\Omega}\, \tilde{g}_{0,j}^{-1}\,d_{0,j}\,
        ( \tilde{h}_{p-1,j}^{-1}\tilde{\Omega} \tilde{g}_{0,j}^{-1})^{-1}$,
the argument is the same.
This complete the proof. 
\end{proof}
%


Next we consider general cases $k\ge 2$. That is,
we consider 
the monodromy relations at  $y=\gamma_j\xi^k$.
Then we take a path $L_k$
which connects $\gamma_0$ and $\gamma_0\xi^k$:
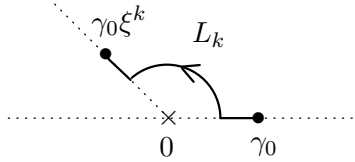
\begin{figure}[H]
\centering
\unitlength 0.1in
\begin{picture}( 19.4000,  6.6400)(  4.6000,-16.6500)
%
\special{pn 8}%
\special{pa 600 1600}%
\special{pa 2400 1600}%
\special{dt 0.045}%
%
\special{pn 13}%
\special{ar 1430 1600 280 280  3.9982983 6.2831853}%
%
\special{pn 13}%
\special{pa 1710 1600}%
\special{pa 1910 1600}%
\special{fp}%
%
\special{pn 13}%
\special{pa 1240 1400}%
\special{pa 1094 1262}%
\special{fp}%
\put(19.4000,-17.5000){\makebox(0,0){$\gamma_0$}}%
\put(19.1000,-16.0000){\makebox(0,0){$\bullet$}}%
\put(11.1000,-12.7000){\makebox(0,0){$\bullet$}}%
\put(11.8000,-10.9000){\makebox(0,0){$\gamma_0 \xi^k$}}%
\put(16.5000,-11.6000){\makebox(0,0){$L_k$}}%
\put(14.3000,-17.5000){\makebox(0,0){$0$}}%
\put(14.4000,-16.0000){\makebox(0,0){$\times$}}%
%
\special{pn 13}%
\special{pa 1510 1340}%
\special{pa 1570 1410}%
\special{fp}%
%
\special{pn 13}%
\special{pa 1510 1330}%
\special{pa 1602 1322}%
\special{fp}%
%
\special{pn 8}%
\special{pa 1442 1602}%
\special{pa 842 1002}%
\special{dt 0.045}%
\end{picture}%
\vspace{0.5cm}
 \caption{The loop $L_k$}
\end{figure}
\noindent
By the exact same   arguments as in the case $k=1$,
 we see   that no new monodromy relations are necessary.

\subsection{The group structures of $\pi_1(\Bbb{C}^2\setminus C)$ and $\pi_1(\Bbb{P}^2\setminus C)$}  
In this section,
we consider the group structures of $\pi_1(\Bbb{P}^2-C)$ and
$\pi_1(\Bbb{C}^2\setminus C)$.
First 
by
previous  considerations,
we have prove that 
\begin{eqnarray}
\pi_1(\Bbb{C}^2\setminus C)=\langle \omega,\, a_{i,j},i=0,\dots,p-1,j=1,\dots,q,
 \,|\, (1{\text{-}}2), (2{\text{-}}3), (S)
 \rangle
\end{eqnarray}
where 
\begin{equation*}
 \begin{split}
(1{\text{-}}2)&\quad a_{i,j}=
\begin{cases}
   \omega^{p-1}a_{i+1,j}\omega^{-(p-1)} & 0\le i\le p-2,\\
   \omega^{2p-1}a_{0,j}(\omega^{2p-1})^{-1} &
  i= p-1,
\end{cases},\quad j=1,\dots,q \\
(2{\text{-}}3)&\quad
a_{i,j}=
       \begin{cases}
       ( g_{i+1,j}h_{i,j})^{-1}\, a_{i+1,j}\, g_{i+1,j}h_{i,j}  & 0\le i\le p-2\\
     h_{p-1,j}^{-1}\,\Omega\, g_{0,j}^{-1}\,a_{0,j}\,
        ( h_{p-1,j}^{-1}\Omega g_{0,j}^{-1})^{-1}
    & i = p-1
        \end{cases},\quad j=1,\dots,q\\
\end{split}
\end{equation*}
\[
 \tag{$S$} \omega=a_{0,1}\cdots a_{0,q}.
\]
Note that last relations in $(1\text{-}2)$ and $(2\text{-}3)$
are unnecessary as they  follow from 
previous relations.  
By the definitions of $g_{i,j}$ and $h_{i,j}$ and $(1\text{-}2)$,
we have the following inductive relations.
\begin{eqnarray}
\begin{cases}
 h_{i+1,j}=\omega^{-(p-1)}h_{i,j}\omega^{p-1},\quad
  g_{i+1,j}=\omega^{-(p-1)}g_{i,j}\omega^{p-1}\\
 g_{i+1,j}h_{i,j}=\omega^{-(p-1)}g_{i,j}\omega^{p-1} h_{i,j}
 \end{cases}
\end{eqnarray}
First we examine the relation
$(2\text{-}3)$ for a fixed $i\le p-2$ using $(1\text{-}2)$.  
The case $j=1$ gives the equality:
\[
  \begin{split}
a_{i,1}&= ( g_{i+1,1})^{-1}\, a_{i+1,1}\, g_{i+1,1},\qquad\qquad \text{as}\,\,
h_{i,1}=e\\
  &=(\omega^{-(p-1)}g_{i,1}\omega^{p-1} )^{-1}
(\omega^{-(p-1)}a_{i,1}\omega^{p-1})
 (\omega^{-(p-1)}g_{i,1}\omega^{p-1} )\\
&=\omega^{-(p-1)}g_{i,1}^{-1}
  a_{i,1}
  g_{i,1}\omega^{p-1}\\
&=\omega^{-p}a_{i,1}\omega^p.
\end{split}
\]
This implies $\omega^p$ and $a_{i,1}$ commute.
Now by the induction on $j$,
we show  that 
\[
\tag{$R_i$} [a_{i,j},\omega^p]=e, \qquad j=1,\dots, q
\]
 where 
$[a,b]=aba^{-1}b^{-1}$.
\begin{proof}
In fact, assuming   $a_{i,1},\dots, a_{i,j-1}$ commute with $\omega^p$,
we get 
    \[
  \begin{split}
a_{i,j}&= ( g_{i+1,j}h_{i,j})^{-1}\, a_{i+1,j}\, 
 g_{i+1,j}h_{i,j}\\
  &=(\omega^{-(p-1)}g_{i,j}\omega^{p-1}h_{i,j} )^{-1}
(\omega^{-(p-1)}a_{i,j}\omega^{p-1})
 (\omega^{-(p-1)}g_{i,j}\omega^{p-1}h_{i,j})\\
&=h_{i,j}^{-1}\omega^{-(p-1)}g_{i,j}^{-1}
  a_{i,j}
  g_{i,j}\omega^{p-1}h_{i,j}\\
&=h_{i,j}^{-1}\omega^{-(p-1)}g_{i,j}^{-1}
  (h_{i,j}^{-1}h_{i,j})a_{i,j}
  g_{i,j}\omega^{p-1}h_{i,j}\\
&=h_{i,j}^{-1}\omega^{-p}h_{i,j}a_{i,j}
  h_{i,j}^{-1}\omega^{p}h_{i,j},\hspace{3cm} \text{as}\ [h_{i,j}, \omega^p ]=e \\
&=\omega^{-p}a_{i,j}\omega^p.
\end{split}
\]
 Thus we get $[a_{i,j},\omega^p]=e$
for all $j=1,\dots,q$.
\end{proof}
The relation $(R_i)$ for $i=0,\dots,p-1$ implies 
$\omega^p$ is in the center of $\pi_1(\Bbb{C}^2\setminus C)$.
Using relations $(1\text{-}2)$ and $(R_i)$, we have
 \begin{equation*}
 \begin{split}
 a_{i+1,j} =\omega a_{i,j}\omega^{-1},
  \quad i=0, \dots, p-2,\ \
        j=1,\dots,  q.  
 \end{split}
\end{equation*}
Thus we get 
\begin{eqnarray}
a_{i,j}=\omega^{i}a_{0,j}\omega^{-i},
\quad i=0, \dots, p-1,\ \
        j=1,\dots,  q.  
\end{eqnarray}
Hence we can take
$a_{0,1},\dots,a_{0,q}$ as generators.
They satisfy the relations
 \[
\tag{$R_0$}
[a_{0,j},\omega^p]=e,\quad j=1,\dots,q. 
 \]
It is easy to see (and we have seen implicitly in the above
discussions) that the relations $(1\text{-}2)$ and $(2\text{-}3)$ follow 
from $(R_0)$, $(S)$ and $(5)$.  
Thus  we have shown  
\[
\begin{split}
 \pi_1(\Bbb{C}^2\setminus C)&=\langle
                       a_{i,j}\,(i=0, \dots, p-1,
        j=1,\dots,  q),  \omega \,|\, (1{\text{-}1}),\
                       (2{\text{-}}3),\
                     (S),\ (5)
                    \rangle\\
                &  \cong
 \langle
                       a_{0,1},\dots, a_{0,q}, \omega \,|\,
 (R_0),\ (S)  \rangle\\
 \pi_1(\Bbb{P}^2\setminus C)& \cong
              \langle
                    a_{0,1},\dots, a_{0,q}, \omega \,|\,\omega^p=e,\
(R_0),\ (S) 
                    \rangle\\
& \cong    \langle
                     a_{0,1},\dots, a_{0,q}, \omega \,|\,\omega^p=e,\ (S)
                     \rangle\\
 &\cong \langle  a_{0,1},\dots, a_{0,q-1}, \omega \,|\,\omega^p=e  \rangle\\
&\cong F(q-1)* \Bbb{Z}/p\Bbb{Z}.
\end{split}
\]
This completes the proof of Theorem 1.

I would like to express my deepest gratitude to 
Professor Mutsuo Oka for
suggesting me the present problem and helping me to prepare this paper.

\def\cprime{$'$} \def\cprime{$'$} \def\cprime{$'$} \def\cprime{$'$}
  \def\cprime{$'$}


\end{document}